\documentclass[a4paper,10pt]{amsart}
\usepackage{amssymb,amsmath,amsfonts,amsthm}
\usepackage[dvips]{graphicx}
\usepackage{psfrag}
\usepackage{todonotes}
\usepackage{color}
\usepackage{enumitem}

\theoremstyle{plain}
\newtheorem{main}{Theorem}

\newtheorem{theorem}{Theorem}[section]
\newtheorem{lemma}[theorem]{Lemma}
\newtheorem{proposition}[theorem]{Proposition}
\newtheorem{corollary}[theorem]{Corollary}
\theoremstyle{remark}
\newtheorem{remark}[theorem]{Remark}

\newtheorem{example}[theorem]{Example}

\newtheorem{conjecture}[theorem]{Conjecture}

\newcommand{\Leb}{\operatorname{vol}}

\newcommand{\Gibbs}{\operatorname{Gibbs}}

\newcommand{\constant}{\operatorname{const}}

\newcommand{\lcs}{\lambda^{cs}}

\newcounter{property}

\newenvironment{property*}[1][]{\par\medskip
	\noindent\textbf{P\theproperty#1'.} \rmfamily}{\medskip}

           \def\ea{\end{array}}
          \def\ec{\end{center}}
     \def\ed{\end{description}}
        \def\ee{\end{equation}}
       \def\eea{\end{eqnarray}}
     \def\eeaa{\end{eqnarray*}}
 \def\et{\end{thebibliography}}

\def\supp{\operatorname{supp}}

\def\cG{{\mathcal G}}

\def\cR{{\mathcal R}}

\def\cH{{\mathcal H}}

\def\cF{{\mathcal F}}
\def\cM{{\mathcal M}}

\def\cR{{\mathcal R}}
\def\cS{{\mathcal S}}
\def\cW{{\mathcal W}}
\def\cY{{\mathcal Y}}

\def\epsilon{\varepsilon}

\def\TT{{\mathbb T}}
\def\RR{{\mathbb R}}

\def\ZZ{{\mathbb Z}}
\def\NN{{\mathbb N}}

\def\MM{\operatorname{MM}}

\title[Maximal transverse measures of expanding foliations]
{
maximal transverse measures of expanding foliations}

\author{Raul Ures, Marcelo Viana, Fan Yang and Jiagang Yang}
\date{\today}

\thanks{R. U. was partially supported by NNSFC 11871262, NNSFC 12071202, and NNSFC 12161141002.
M.V. and J.Y. were partially supported by CNPq, FAPERJ, and PRONEX.
We acknowledge support from the Fondation Louis D--Institut de France (project coordinated by M. Viana).}

\address{1. Department of Mathematics, Southern University of Science and Technology, Shenzhen, Guangdong, China
\newline
\hspace*{0.45cm}2. SUSTech International Center for Mathematics, Shenzhen, Guangdong, China
}
\email{ures\@@sustech.edu.cn}

\address{IMPA, Est. D. Castorina 110, 22460-320 Rio de Janeiro, Brazil}
\email{viana\@@impa.br}

\address{Department of Mathematics, Wake Forest University, Winston-Salem, USA.}
\email{yangf\@@wfu.edu}

\address{Departamento de Geometria, Instituto de Matem\'atica e Estat\'\i stica, Universidade Federal Fluminense, Niter\'oi, Brazil}
\email{yangjg\@@impa.br}

\begin{document}

\begin{abstract}
For an expanding (unstable) foliation of a diffeomorphism, we use a natural dynamical averaging
to construct transverse measures, which we call \emph{maximal}, describing the statistics of how
the iterates of a given leaf intersect the cross-sections to the foliation.
For a suitable class of diffeomorphisms, we prove that this averaging converges,
even exponentially fast, and the limit measures have finite ergodic decompositions.
These results are obtained through relating the maximal transverse measures to the maximal
$u$-entropy measures of the diffeomorphism (see~\cite{UVYY1}).
\end{abstract}

\maketitle

\tableofcontents

\section{Introduction}\label{s.introduction}

The (Kolmogorov--Sinai~\cite{Kol58,Sin59}) entropy of a smooth deterministic system, such as a diffeomorphism $f:M \to M$,
is a parameter of the complexity of its evolution from the perspective of arbitrary discretizations of the phase-space.
Namely, the entropy $h_\mu(f)$ with respect to any stationary probability distribution $\mu$ is the average Shannon information associated to itineraries relative to arbitrarily fine finite partitions of $M$. That such information does arise from the time
evolution is because the system fails to be deterministic once it is discretized: itineraries are not determined by the initial partition element. In that sense, the entropy is intimately related to the divergence of trajectories associated to nearby initial conditions. 

The \emph{variational principle} (Dinaburg~\cite{Din70,Din71}, Goodman~\cite{Gdm71}, Goodwin~\cite{Gdw71}) states that, under rather general conditions, the supremum of $h_\mu(f)$ over all stationary probability distributions coincides with the system's topological entropy $h(f)$. The latter notion, which was introduced by Adler, Konheim, McAndrew~\cite{AKM65}, mimics the concept of entropy except that one considers ``itineraries'' relative to finite covers of the phase-space by arbitrarily small open sets.
When they exist, stationary probability distributions realizing the supremum are called \emph{measures of maximal entropy}. 

In some important situations, divergence of trajectories takes place \emph{preferably} along certain well-defined directions which
define an invariant \emph{lamination}, that is, an invariant decomposition of the phase-space into smooth submanifolds
(the \emph{leaves} of the lamination).
In such situations, one is naturally to consider the  corresponding \emph{partial entropy}, which measures the complexity of the dynamics along such an invariant structure. This idea goes back to Pesin~\cite{Pes77}, Ledrappier~\cite{Led84a}, and
Ledrappier, Young~\cite{LeY85a,LeY85b}, where the relevant invariant structure is the Pesin unstable lamination.

An important case which is more directly relevant to the present paper, is that of the strong-unstable foliation $\cF^{uu}$ of a
partially hyperbolic diffeomorphism $f:M\to M$: the leaves are tangent at every point to the subspace $E^{uu}_x$ along which
the system diverges most strongly. We denote by $h_\mu(f,\cF^{uu})$ the corresponding partial entropy, which we call the 
\emph{$u$-entropy}.

A notion of \emph{partial topological entropy} was also introduced in this setting by Saghin, Xia~\cite{SaX09}. We call it the \emph{topological $u$-entropy} of the diffeomorphism, and denote it as $h(f,\cF^{uu})$.
More recently, Hu, Hua, Wu~\cite{HHW17} proved the variational $u$-principle: $h(f,\cF^{uu})$ coincides with the supremum
of $h_\mu(f,\cF^{uu})$ over all invariant probability distributions $\mu$. Besides, \emph{measures of maximal $u$-entropy},
for which the supremum is attained, always exist in this context.

The main goal of this paper is to point out that measures of maximal $u$-entropy are a tool to describe the behavior
of strong-unstable foliations and, especially, the transverse behavior of their leaves. In doing this we will unveil
certain interesting connections between three different kinds of transverse objects:
\begin{itemize}
\item hitting measures of leaf iterates with cross-sections, carrying topological and dynamical information;
\item quotients of measures of maximal $u$-entropy under holonomy maps, originating from ergodic theory;
\item invariant transverse measures, arising from foliation theory (Plante~\cite{Pla75}, Goodman, Plante~\cite{GoP79}).
\end{itemize}

Before explaining these connections formally, let us illustrate them through a simple example.
This is a variation the classical solenoid of Smale~\cite{Sma67} which has been considered a number of times (see
\cite[Section~8]{UVYY1} and references therein).
Let $D$ be the 2-dimensional disk, and $f_0:S^1\times D \to S^1 \times D$ be a $C^2$ embedding of the form
\begin{equation}\label{eq.topologicalsolnoid}
f_0(\theta, x) = (k\theta \operatorname{mod}\ZZ, h_\theta(x))
\end{equation}
where $k \ge 3$ and $h_\theta$ is such that $\|Dh_\theta^{\pm 1}\|$ are strictly less than $k$ at every point.
Assume also that $\|Dh_0\| < 1$ where $0 \in S^1$ is a fixed point for $\theta \mapsto k\theta \operatorname{mod} \ZZ$.
Using the theory developed in Sections~8 and~9 of \cite{UVYY1} one can check that $f_0$ has a unique measure $\mu$ of
maximal $u$-entropy.

Let $\Lambda_0 = \cap_{n\ge 1} f_0^n(M)$ be the corresponding attractor. This is known (see \cite{UVYY1}) to be a
disjoint union of smooth curves, called \emph{strong-unstable leaves}, which form an invariant lamination of the
attractor transverse to every ``vertical'' disk $S_\theta = \{\theta\} \times D$.
We call \emph{strong-unstable plaque} any connected component $\xi$ of a strong-unstable leaf minus the disk $S_0$.

The iterates $f^n(\xi)$ of any strong-unstable plaque grow exponentially fast as $n$ increases, becoming denser and
denser in the attractor $\Lambda_0$. A natural way to try and describe this process quantitatively is by looking at
how those iterates hit each of the vertical disks $S_\theta$. More precisely, we consider the \emph{hitting measures} 
$$
\hat\mu_{\theta,n} = \frac{1}{\#\{z \in S_\theta \cap f^n(\xi)\}}
\sum_{z \in S_\theta \cap f^n(\xi)} \delta_z
$$
of $f^n(\xi)$ with each $S_\theta$.

Our results show that the sequence $(\hat\mu_{\theta,n})_n$ converges to some probability measure $\hat\mu_\theta$ on the
disk $S_\theta$ which does not depend on the choice of the strong-unstable plaque.
The family $\{\hat\mu_\theta: \theta \in S^1\}$ is a \emph{invariant transverse measure} for the strong-unstable lamination,
in the sense of Plante~\cite{Pla75,GoP79} and Bowen, Marcus~\cite{BoM77}: invariance means that the projections (holonomies)
along strong-unstable leaves map the $\hat\mu_\theta$ to one another. 
See Figure~\ref{fig:solenoid}.

\begin{figure}[ht]
\psfrag{S1}{$S_\theta$}\psfrag{S2}{$S_{\theta'}$}
\psfrag{X}{$f^n(\xi)$}
\centering
\includegraphics[width=3in]{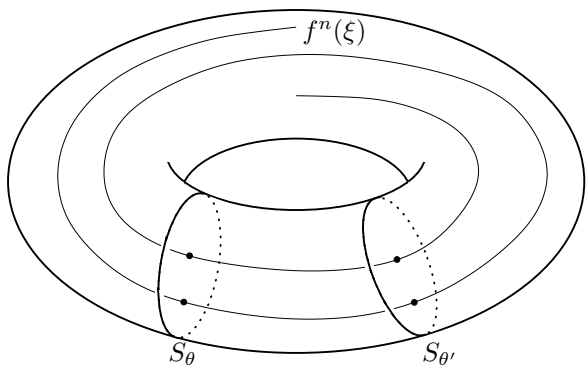}
\caption{\label{fig:solenoid}Hitting measures record the intersections of iterates $f^n(\xi)$ of strong-unstable plaques with each of the transverse disks $S_\theta$, $\theta\in S^1$. As $n$ goes to infinity, they converge \emph{exponentially fast} to a family $\hat\mu_\theta$ of probability measures on the transverse disks, which is invariant under the holonomies of the strong-unstable lamination.}
\end{figure}

Now, it turns out the limit invariant transverse measure is completely determined by the measure $\mu$ of maximal $u$-entropy.
Namely, each $\hat\mu_\theta$ coincides with the projection of $\mu$ to $S_\theta$ along strong-unstable leaves
(let us say that we always project in the counterclockwise direction). Thus we call it the \emph{maximal transverse measure}.
Moreover, surprisingly, the convergence $\hat\mu_{\theta,n} \to \hat\mu_\theta$ is exponentially fast, in a natural sense.
The precise statements will appear in the next section.

Similar statements hold for a much more general class of dynamical systems, as we are going to see, and may actually be typical
in great generality amongst partially hyperbolic systems. For instance, see Conjecture~\ref{cj.time-1} below.

\paragraph{\textbf{Acknowledgements}} We are grateful to the anonymous reviewers for several comments and corrections which greatly helped us improve the presentation.

\section{Statement of main results}\label{s.statement}

Throughout this paper, we take $f:M\to M$ to be a partially hyperbolic $C^1$ diffeomorphism
which factors over an Anosov automorphism $A:\TT^d\to\TT^d$.
We start by defining these concepts precisely.

A $C^1$ diffeomorphism $f:M \to M$ on a compact manifold $M$ is \emph{partially hyperbolic}
if there exists a $Df$-invariant splitting
$$
T M = E^{cs} \oplus E^{uu}
$$
of the tangent bundle such that $Df \mid_{E^{uu}}$ is uniformly expanding and \emph{dominates}
$Df \mid_{E^{cs}}$. By this we mean that there exist a Riemmanian metric on $M$ and a
constant $\omega < 1$ such that
\begin{equation}\label{eq.omega}
\|Df(x)v^{uu}\| \ge \frac{1}{\omega}
\text{ and }
\frac{\|Df(x)v^{cs}\|}{\|Df(x)v^{uu}\|}
\le \omega
\end{equation}
for any unit vectors $v^{cs}\in E^{cs}_x$ and $v^{uu}\in E^{uu}_x$,
and any $x \in M$.

The \emph{strong-unstable} sub-bundle $E^{uu}$ is uniquely integrable, meaning that there exists a
unique foliation $\cF^{uu}$ which is invariant under $f$ and tangent to $E^{uu}$ at every point.
The corresponding holonomy maps are H\"older continuous:
see \cite[Corollary~2.1]{BP74} and \cite[Theorem~6.4]{HP70}, which use a stronger (absolute)
partial hyperbolicity condition; a proof for the (pointwise) notion of partial hyperbolicity we
assume here can be found in \cite{hch_phd}. See also \cite[Theorem~A']{PSW97}.

We assume that $f$ is \emph{dynamically coherent}, meaning that the \emph{center-stable} sub-bundle $E^{cs}$
admits some $f$-invariant integral foliation $\cF^{cs}$. See \cite{HPS77}.
A compact $f$-invariant set $\Lambda\subset M$ is \emph{$u$-saturated} if it consists of entire leaves of $\cF^{uu}$.
Then it is called \emph{$u$-minimal} if every strong-unstable leaf contained in $\Lambda$ is dense in $\Lambda$.

As in  \cite{UVYY1}, we say that $f$ \emph{factors over Anosov} if there exist a hyperbolic linear automorphism
$A:\TT^d\to\TT^d$ on some torus, and a continuous surjective map $\pi:M\to\TT^d$ such that
\begin{enumerate}
\item[(H1)] $\pi \circ f = A \circ \pi$;
\item[(H2)] $\pi$ maps each strong-unstable leaf of $f$ homeomorphically to an unstable leaf of $A$;
\item[(H3)] $\pi$ maps each center-stable leaf of $f$ to a stable leaf of $A$.
\end{enumerate}
Several examples of diffeomorphisms that factor over Anosov were described in \cite[Sections 6 to 9]{UVYY1}.

Let $\cR=\{\cR_1, \dots, \cR_k\}$ be a Markov partition for the linear automorphism,
and $\cM=\{\cM_1, \dots, \cM_k\}$ be defined by $\cM_i=\pi^{-1}(\cR_i)$.
We denote by $\cW^s$ and $\cW^u$, respectively, the stable and unstable foliations of $A$.
The connected components of the intersection of their leaves with each
element of $\cR$ will be called \emph{stable/unstable plaques} of $A$.
Center-stable plaques and strong-unstable plaques of $f$ are defined analogously,
considering the leaves of $\cF^{cs}$ and $\cF^{uu}$ instead.

Now let us state our main results.
Some of the technical notions involved in the statements will be defined along the way
(see Section~\ref{s.transverse_invariant} for the precise definition of an invariant transverse measure).

\begin{main}[Existence of transverse measures]\label{main.A}
Let $f:M \to M$ be a partially hyperbolic $C^1$ diffeomorphism that factors over Anosov.
Then there exist positive constants $c_1, \dots, c_k$ such that,
for any $f$-invariant measure $\mu$ of maximal $u$-entropy,
$\tau_\mu=\{c_i\hat\mu_S: S \subset\cM_i\}$ is an invariant transverse measure
for the strong-unstable foliation $\cF^{uu}$,
where $S$ denotes a cross-section to the strong-unstable foliation, 
and $\hat\mu_S$ is the projection of $\mu \mid \cM_i$ to $S$ along the
strong-unstable plaques.
\end{main}

A partially hyperbolic diffeomorphism $f:M \to M$ is said to have \emph{$c$-mostly contracting center}
if the center Lyapunov exponents of every ergodic measure of maximal $u$-entropy are negative.
This is a variation of the notion of diffeomorphism with mostly contracting center,
which was introduced in \cite{BoV00} and has been developed by several authors,
for instance in~\cite{Cas02,Dol00,DVY16}.
Examples and more information on systems with $c$-mostly contracting center,
including alternative equivalent definitions, can be found in \cite{UVYY1} and
Section~\ref{s.proofC} below.

\begin{main}[Exponential convergence]\label{main.B}
Let $f:M\to M$ be a partially hyperbolic diffeomorphism that factors over Anosov and has $c$-mostly contracting center.
Let $\mu$ be an ergodic measure of maximal $u$-entropy whose support is connected.
Then for any strong-unstable plaque $\xi^u(x)$ with $x\in\supp\mu$,
and any cross-section $S$ contained in some $\cM_j$ with $\mu(\cM_j)>0$,
and any H\"older real function $\hat\varphi$ supported inside $S$,
$$
\frac{1}{\# \big(f^n(\xi^u(x)) \cap S\big)} \sum_{q \in f^n(\xi^u(x)) \cap S} \hat\varphi(q) 
\to \frac{1}{\|\hat\mu_S\|} \int \hat\varphi \, d\hat\mu_S
$$
exponentially fast as $n\to\infty$.
\end{main}

\begin{example}
Let $f:M \to M$ be any partially hyperbolic diffeomorphism on a 3-dimensional nilmanifold $M \neq \TT^3$.
Then, by Hammerlindl, Potrie~\cite{HaP14}, $f$ factors over Anosov.
Moreover, if $f$ is $C^{1+}$ and admits some hyperbolic periodic point then  any diffeomorphism in a
$C^1$-neighborhood has $c$-mostly contracting center. Thus Theorems~\ref{main.A} and~\ref{main.B} apply to it.
Indeed, in this case there is exactly one measure of maximal $u$-entropy, and its support is connected.
See \cite[Section~7]{UVYY1}. 
\end{example}

\begin{remark}\label{r.not_connected}
A version of Theorem~\ref{main.B} remains true when the support of $\mu$ is not connected.
That is because the support has finitely many connected components, and the normalized
restrictions of $\mu$ to each connected component are ergodic for a suitable iterate $f^l$.
See Lemma~4.9 and equation (28) in \cite{UVYY1}. Then the previous statement can be applied to
the restriction of $f^l$ to each connected component. Corresponding observations apply to
Theorem~\ref{main.C}, Theorem~\ref{main.D} and Theorem~\ref{t.large_deviations} below.
\end{remark}

Given any $f$-invariant probability measure $\mu$, we say that the system $(f,\mu)$ has \emph{exponential decay
of correlations} for H\"{o}lder observables if for any $\gamma \in(0,1]$ there exists $\tau<1$ such that for
all $\gamma$-H\" older functions $\varphi,\psi: M \to \RR$ there exists $K(\varphi,\psi)>0$ such that
$$
\left| \int (\varphi\circ f^n)\psi\,d\mu - \int \varphi\,d\mu \int \psi\,d\mu \right|
\leq K(\varphi,\psi) \tau^n \text{ for every }n\geq 1.
$$
The technical heart of the proof of Theorem~\ref{main.B} is the coupling lemma (Lemma~\ref{l.coupling})
which we prove in Section~\ref{s.coupling2}, based on methods developed by Young~\cite{You99} and
Dolgopyat~\cite{Dol00}. This lemma also yields the following result:

\begin{main}\label{main.C}
Let $f:M\to M$ be a partially hyperbolic diffeomorphism that factors over Anosov and has
$c$-mostly contracting center. Let $\mu$ be an ergodic measure of maximal $u$-entropy whose support
$\supp\mu$ is connected. Then $(f,\mu)$ has exponential decay of correlations for H\"{o}lder observables.
\end{main}

For the next theorem we assume that the center-stable bundle admits a dominated splitting
$E^{cs}= E^{ss} \oplus E^{c}$, where $E^{ss}$ is uniformly contracting.
Thus, $f$ admits a partially hyperbolic splitting $E^{ss} \oplus E^{c} \oplus E^{uu}$
into three sub-bundles, meaning that relative to some Riemmanian metric on $M$ there exists
$\omega < 1$ such that
\begin{equation}\label{eq.omega2}
\begin{aligned}
\|Df(x)v^{uu}\| \ge \frac{1}{\omega},
\quad
\|Df(x)v^{ss}\| \le \omega, \text{ and }\\
\frac{1}{\omega} \|Df(x)v^{ss}\|
\le \|Df(x)v^{c}\|
\le \omega \|Df(x)v^{uu}\|
\end{aligned}
\end{equation}
for any unit vectors $v^{ss}\in E^{ss}_x$, $v^{c}\in E^{c}_x$ and $v^{uu}\in E^{uu}_x$,
and any $x \in M$.
We denote by $\cF^{ss}$ the \emph{strong-stable} foliation, that is, the unique foliation
tangent to $E^{ss}$ at every point.
The assumption that $f$ is dynamically coherent implies that there exists a \emph{center foliation}
$\cF^c$ tangent to $E^c$ at every point.

\begin{main}(Ergodicity)\label{main.D}
Let $f:M\to M$ be a $C^1$ diffeomorphism with a partially hyperbolic splitting $E^{ss} \oplus E^{c} \oplus E^{uu}$
into three sub-bundles. Assume that $f$ factors over Anosov and has $c$-mostly contracting center.
Assume furthermore that the map $\pi:M\to\TT^d$ is a fiber bundle whose fibers are the center leaves of $f$,
and those fibers are compact. 
Then, for any ergodic measure of maximal $u$-entropy $\mu$, the invariant transverse measure
$\tau_\mu$ has a finite ergodic decomposition.
\end{main}

A natural question is whether this theory can be extended to other classes of partially hyperbolic
diffeomorphisms. In this direction we state:
\begin{conjecture}\label{cj.time-1}
The conclusions of Theorem~\ref{main.A}, Theorem~\ref{main.B} and Theorem~\ref{main.C} remain  
true for perturbations of time-1 maps of transitive Anosov flows.
\end{conjecture}

This paper is organized as follows. In Section~\ref{s.preliminaries1} we recall several useful notions and
facts, mostly from our previous paper \cite{UVYY1}. We also recall the concept of invariant transverse measure,
which we take from Bowen, Marcus~\cite{BoM77}. Section~\ref{s.proofA} and Section~\ref{s.proofD} contain the
proofs of Theorem~\ref{main.A} and Theorem~\ref{main.D}, respectively.
In Section~\ref{s.proofB} we reduce the proof of Theorem~\ref{main.B} to a key technical result which is
stated in Lemma~\ref{l.coupling}.

The proof of Lemma~\ref{l.coupling} is based on the coupling argument of Young~\cite{You99},
in the form developed by Dolgopyat~\cite[Sections~6--9]{Dol00} for diffeomorphisms with mostly contracting
center, which we present in Section~\ref{s.coupling2}. In preparation for that argument,
in Section~\ref{s.large_deviations} we prove a large deviations principle (Theorem~\ref{t.large_deviations})
which is interesting in itself. Lemma~\ref{l.coupling} is also the key step in the proof of Theorem~\ref{main.C},
which we present in Section~\ref{s.proofC}.

\section{Preliminaries}\label{s.preliminaries1}

The \emph{topological $u$-entropy} of $f$, denoted by $h(f,\cF^{uu})$, is the maximal rate of volume growth
for any disk contained in a strong-unstable leaf. See Saghin, Xia~\cite{SaX09}.
The \emph{$u$-entropy} of an $f$-invariant measure $\mu$, denoted as $h_\mu(f,\cF^{uu})$,
is defined by
$$
h_\mu(f,\mu) = H_\mu\left(f^{-1} \xi^u \mid \xi^u\right)
$$
where $\xi^u$ is any measurable partition subordinate to the strong-unstable foliation.
Recall that, according to Rokhlin~\cite[Section~7]{Rok67a}, the entropy $h_\mu(f)$ is the
supremum of $H_\mu\left(f^{-1} \xi \mid \xi\right)$ over all measurable partitions $\xi$
with $f^{-1}\xi \prec \xi$. Thus we always have
\begin{equation}\label{eq.two_entropies}
h_\mu(f,\cF^{uu}) \le h_\mu(f).
\end{equation}
See Ledrappier, Strelcyn~\cite{LeS82}, Ledrappier~\cite{Led84a},  Ledrappier, Young~\cite{LeY85a},
and Yang~\cite{Yan21}. We call $\mu$ a \emph{measure of maximal $u$-entropy} if it satisfies
$$
h_\mu(f,\cF^{uu})=h(f,\cF^{uu}).
$$
By Hu, Wu, Zhu~\cite{HWZ21}, the set $\MM^u(f)$ of measures of maximal $u$-entropy is always
non-empty, convex and compact. Moreover, its extreme points are ergodic measures.

\subsection{Markov partitions}\label{s.Markov.partitions}

Let $\cR=\{\cR_1, \dots, \cR_k\}$ be a Markov partition for the linear automorphism
$A:\TT^d\to\TT^d$. By this we mean (see Bowen~\cite[Section~3.C]{Bow75a})
a finite covering of $\TT^d$ by small closed subsets $\cR_i$ such that
\begin{enumerate}
\item[(a)] each $\cR_i$ is the closure of its interior, and the interiors are pairwise disjoint;
\item[(b)] for any $a, b \in\cR_i$, $W^u_i(a)$ intersects $W_i^s(b)$ at exactly one point,
which we denote as $[a,b]$;
\item[(c)] $A(W^s_i(a)) \subset W^s_j(A(a))$ and $A(W^u_i(a)) \supset W^u_j(A(a))$ if $a$ is in
the interior of $\cR_i$ and $A(a)$ is in the interior of $\cR_j$.
\end{enumerate}
Here, $W^u_i(a)$ is the connected component of $W^u(a) \cap \cR_i$ that contains $a$,
and $W^s_i(a)$ is the connected component of $W^s(a) \cap \cR_i$ that contains $a$.
We call them, respectively, the \emph{unstable plaque} and
the \emph{stable plaque} through $a$.
Property (b) is called \emph{local product structure}.

The boundary $\partial\cR_i$ of each $\cR_i$ coincides with
$\partial^s\cR_i \cup \partial^u\cR_i$, where $\partial^s\cR_i$ is the set of points
$x$ which are not in the interior of $W^u_i(x)$ inside the corresponding unstable leaf,
and $\partial^u\cR_i$ is defined analogously.
By product structure, $\partial^s\cR_i$ consists of stable plaques and
$\partial^u\cR_i$ consists of unstable plaques.
The Markov property (c) implies that the total stable boundary
$\partial^s\cR = \cup_i \partial^s\cR_i$ is forward invariant and the
total unstable boundary
$\partial^u\cR = \cup_i \partial^u\cR_i$ is backward invariant under $A$.
Since the Lebesgue measure on $\TT^d$ is invariant and ergodic for $A$,
it follows that both $\partial^s\cR$ and $\partial^u\cR$ have zero
Lebesgue measure. Then, by Fubini, the intersection of $\partial^s\cR$
with almost every unstable plaque has zero Lebesgue measure in the plaque.
It follows that the same is true for \emph{every} unstable plaque,
since the stable holonomies of $A$, being affine, preserve the class of
sets with zero Lebesgue measure inside unstable leaves.
A similar statement holds for $\partial^u\cR$.

Next, define $\cM = \{\cM_1,\ldots,\cM_k\}$ by $\cM_i = \pi^{-1}(\cR_i)$.
For each $i=1, \dots, k$ and $x\in\cM_i$, let $\xi_i^u(x)$ be the connected component of
$\cF^{uu}(x)\cap\cM_i$ that contains $x$, and $\xi_i^{cs}(x)$ be the pre-image of $W^{s}_i(\pi(x))$. By construction,
\begin{equation}\label{eq.Markov1}
f(\xi_i^{u}(x)) \supset \xi_j^{u}(f(x))
\text{ and }
f(\xi_i^{cs}(x)) \subset \xi_j^{cs}(f(x))
\end{equation}
whenever $x$ is in the interior of $\cM_i$ and $f(x)$ is in the interior of $\cM_j$.
We refer to $\xi^i_u(x)$ and $\xi^i_{cs}(x)$, respectively,
as the \emph{strong-unstable plaque} and the \emph{center-stable plaque}
through $x$.

The local product structure property also extends to $\cM$:
for any $x, y \in \cM_i$ we have that $\xi^u_i(x)$ intersects $\xi^{cs}_i(y)$ at exactly
one point, which we still denote as $[x,y]$. That can be seen as follows.
To begin with, we claim that $\pi$ maps $\xi_i^u(x)$ homeomorphically to $W^u_i(\pi(x))$.
In view of the assumption (H2) above, to prove this it is enough to check that
$\pi(\xi_i^u(x)) = W^u_i(\pi(x))$. The inclusion $\subset$ is clear, as both sets are connected.
Since $\xi_i^u(x)$ is compact, it is also clear that $\pi(\xi_i^u(x))$ is closed in $W^u_i(\pi(x))$.
To conclude, it suffices to check that it is also open in $W^u_i(\pi(x))$.
Let $b=\pi(z)$ for some $z\in\xi_i^u(x)$. By assumption (H2), for any small neighborhood $V$ of $b$
inside $W^u(b)$, there exists a small neighborhood $U$ of $z$ inside $\cF^{uu}_z$
that is mapped homeomorphically to $V$. By definition, a point $w\in U$ is in $\cM_i$ if and
only if $\pi(w)$ is in $\cR_i$. Thus $\pi$ maps $U\cap\cM_i$ homeomorphically to $V\cap\cR_i$.
That implies that $b$ is in the interior of $W^u_i(b)$, and that proves that $\pi(\xi_i^u(x))$
is indeed open in $W^u_i(\pi(x))$. Thus the claim is proved.
Finally, $[x,y]$ is precisely the sole pre-image of $[\pi(x),\pi(y)]$ in $\xi^u(x)$;
notice that this pre-image does belong to $\xi_i^{cs}(y)$, by definition.

This shows that $\cM$ is a Markov partition for $f$, though not necessarily a generating one.
In any event, the fact that $f$ is uniformly expanding along strong-unstable leaves ensures that
$\cM$ is automatically \emph{$u$-generating}, in the sense that
\begin{equation}\label{eq_u-generating}
\bigcap_{n=0}^\infty f^{-n}\left(\xi_i^u\left(f^n(x)\right)\right)
= \{x\} \text{ for every } x\in\Lambda.
\end{equation}
We call \emph{center-stable holonomy} the family of maps $H^{cs}_{x,y}:\xi_i^u(x) \to \xi^u_i(y)$
defined by the condition that
$$
\xi_i^{cs}(z) = \xi_i^{cs}(H^{cs}_{x,y}(z))
$$
whenever $x, y \in \cM_i$ and $z \in \xi_i^u(x)$.

\subsection{Reference measures}

By pulling the Lebesgue measure along the unstable leaves of $A$ back under the factor map $\pi$,
one obtains a special family of measures on the strong-unstable plaques of $f$ that we call the
\emph{reference measures}. More precisely, the reference measures are the probability measures
$\nu^u_{i,x}$ defined on each strong-unstable plaque $\xi_i^u(x)$, $x\in\cM_i$,
$i\in\{1, \dots, k\}$ by
$$
\pi_*\nu^u_{i,x} = \Leb^u_{i,\pi(x)} = \text{normalized Lebesgue measure on } W^u_i(\pi(x)).
$$
Since the Lebesgue measure on unstable leaves are preserved by the stable holonomy of $A$ (as the
latter is affine), the construction in the previous section also gives that these reference measures
are preserved by center-stable holonomies of $f$:
\begin{equation}\label{eq_cs-invariant}
\nu_{i,y}^u = \left(H^{cs}_{x,y}\right)_*\nu^u_{i,x}.
\end{equation}
for every $x$ and $y$ in the same $\cM_i$. Similarly, the fact that Lebesgue measure on unstable
leaves has constant Jacobian for $A$ implies that the same is true for the reference measures of $f$:
if $f(\cM_i)$ intersects the interior of $\cM_j$ then
\begin{equation}\label{eq_constant_Jacobian}
f_*\left(\nu^u_{i,x} \mid_{f^{-1}\left(\xi_j^u(f(x))\right)}\right)
= \nu^u_{i,x}\left(f^{-1}\left(\xi_j^u(f(x))\right)\right) \nu^u_{j,f(x)}
\end{equation}
for every $x\in\cM_i\cap f^{-1}(\cM_j)$.

\begin{remark}\label{r.lowboundary}
Properties \eqref{eq_cs-invariant} and \eqref{eq_constant_Jacobian} imply that
$x \mapsto \nu^u_{i,x}(f^{-1}\xi_j^u(f(x)))$ is constant on $\cM_i \cap f^{-1}(\cM_j)$,
for any $i$ and $j$ such that $f(\cM_i)$ intersects the interior of $\cM_j$.
Thus this function takes only finitely many values.
\end{remark}

\begin{remark}\label{r.equivalentmeasures}
Let $x$ be on the boundary of two different Markov sets $\cM_i$ and $\cM_j$.
Then the restrictions of $\nu^u_{i,x}$ and $\nu^u_{j,x}$ to the intersection
$\xi^u_i(x) \cap \xi^u_j(x)$ are equivalent measures,
as they are both mapped by $\pi_*$ to multiples of the Lebesgue measure on
$W^u_i(\pi(x)) \cap W^u_j(\pi(x))$.
\end{remark}

\begin{remark}\label{r.zeroboundary}
As observed before, the intersection of $\partial^s\cR$ with every
unstable plaque $W^u_i(x)$ has zero Lebesgue measure inside the plaque.
Since $\pi$ sends each $\cM_i$ to $\cR_i$, with each $\xi^u(x)$ mapped homeomorphically to $W^u_i(\pi(x))$,
it follows that $\partial^s\cM \cap \xi^u_i(x)$ has zero $\nu^u_{i,x}$-measure for every $x\in\cM_i$ and every $i$.
\end{remark}

\subsection{$c$-Gibbs $u$-states}

An $f$-invariant probability measure $\mu$ is called a \emph{$c$-Gibbs $u$-state}
if its conditional probabilities along strong-unstable leaves coincide with
the family of reference measures $\nu_{i,x}^u$.
More precisely, for each $i$, let $\{\mu^u_{i,x}: x \in \cM_i\}$ denote the disintegration of
the restriction $\mu\mid_{\cM_i}$ relative to the partition $\{\xi_i^u(x): x\in \cM_i\}$.
Then we call $\mu$ a \emph{$c$-Gibbs $u$-state} if $\mu^u_{i,x}=\nu_{i,x}^u$ for $\mu$-almost every $x$.
The space of invariant $c$-Gibbs $u$-states of $f$ is denoted by $\Gibbs^u_c(f)$.

\begin{proposition}[Corollary~3.7 and Proposition~4.1 in \cite{UVYY1}]\label{p.gibbs}
$\Gibbs^u_c(f)$  is non-empty, convex, and compact. Moreover,
\begin{enumerate}
\item almost every ergodic component of any $\mu \in \Gibbs^u_c(f)$
is a $c$-Gibbs $u$-state;
\item the support of every $\mu \in \Gibbs^u_c(f)$ is $u$-saturated;
\item for every $x \in \cM_i$ and $l \in\{1, \dots, k\}$,
every accumulation point of the sequence
$$
\mu_n = \frac{1}{n}\sum_{j=0}^{n-1} f^j_* \nu^u_{l,x}
$$
is a $c$-Gibbs $u$-state.
\item an $f$-invariant probability measure $\mu$
is a measure of maximal $u$-entropy if and only if it is a $c$-Gibbs $u$-state.

\item the union $\bigcup_{i=1}^k \partial \cM_i$ of the boundary sets has measure zero
with respect to every $c$-Gibbs $u$-state.
\end{enumerate}
\end{proposition}

\subsection{$c$-mostly contracting center}

We say that $f$ has \emph{$c$-mostly contracting center} if
\begin{equation}\label{eq.c_mostly_contracting}
\limsup_n \frac 1n \log \|Df^n \mid_{E^{cs}}\| < 0
\end{equation}
on a positive measure subset relative to every reference measure.
See \cite{UVYY1}.
The following proposition, together with part (4) of Proposition~\ref{p.gibbs}
shows that this is equivalent to the definition given in the Introduction.

\begin{proposition}[Proposition~4.2 of \cite{UVYY1}]\label{p.contracting}
$f$ has $c$-mostly contracting center if and only if all center-stable
Lyapunov exponents of every ergodic $c$-Gibbs $u$-state of $f$ are negative.
\end{proposition}

\subsection{Invariant transverse measures}\label{s.transverse_invariant}

By a \emph{transverse measure}\footnote{All transverse measures are taken
to be not identically zero, unless stated otherwise.} of a foliation $\cF$,
we mean a family $\tau=\{\hat\mu_S: S \in \cS\}$ of measures, where $\cS$
is a family of small cross-sections $S$ to the foliation such that
\begin{itemize}
\item every $x \in M$ belongs to some $S\in\cS$;
\item if $S \in \cS$ and $S'\subset S$ is a measurable subset then $S'\in\cS$; 
\end{itemize}
and each $\mu_S$ is a non-negative Borel measure on the corresponding cross-section $S$.
We call the transverse measure \emph{invariant} if every holonomy homeomorphism
$h:S_1 \to S_2$ of $\cF$ between cross-sections $S_1, S_2 \in\cS$ maps the measure
$\hat\mu_{S_1}$ to the measure $\hat\mu_{S_2}$.
This follows Bowen, Marcus~\cite{BoM77}, where this notion is called 
\emph{$\cG$-invariant measure}.

The family of (invariant) transverse measures is preserved by any re-scaling
$\hat\mu_S\mapsto c\hat\mu_S$ with $c$ independent of $S$.
It is also clear that if $\tau_1$ and $\tau_2$ are (invariant) transverse measures
defined on the same family $\cS$ of cross-sections, then $\tau_1+\tau_2$ is again an
(invariant) transverse measure.

We call an invariant transverse measure $\tau =\{\hat\mu_S: S \in \cS\}$ \emph{ergodic}
if given any splitting $\tau = \tau' + \tau''$ as a sum of two invariant transverse measures $\tau'$
and $\tau''$ there exists $c \in (0,1)$ such that $\tau' = c\tau$.
We say that an invariant transverse measure  $\tau$ has \emph{finite ergodic decomposition}
if it is a sum of finitely many ergodic invariant transverse measures.

\begin{lemma}\label{l.ergodictransverse1}
If $\tau_1$ and $\tau_2$ are ergodic invariant transverse measures which are non-singular
restricted to some cross-section then there is $c>0$ such that $\tau_2=c\tau_1$.
\end{lemma}

\begin{proof}
For each cross-section $S$, let $\hat\mu_{1,S}$ and $\hat\mu_{2,S}$ be the measures defined on $S$ by $\tau_1$ and $\tau_2$, respectively.
Given any rational numbers $p < q$, let $[p,q]_S \subset S$ be the set of points where the Radon-Nykodim derivatives
satisfies $d\hat\mu_{2,S}/d\hat\mu_{1,S} \in [p,q]$. 
Observe that the Radon-Nykodim derivatives are invariant under holonomies, since the measures themselves are.
Thus, the family of sets $[p,q]_S$ is invariant under holonomy, and so the families of restrictions of the
$\hat\mu_{1,S}$ and $\hat\mu_{2,S}$ to the $[p,q]_S$ define invariant transverse measures $\tau_1'$ and $\tau_2'$
such that $\tau_1' \le \tau_1$ and $\tau_2' \le \tau_2$.
We claim that either $\tau_1'$ vanishes or $\tau_1'=\tau_1$: otherwise, by ergodicity, there would exist $c \in (0, 1)$
such that $\tau_1'=c\tau_1$, and this is not possible because $\tau_1'$ is a restriction of $\tau_1$.
This ensures that, given any $p < q$, the set $[p,q]_S$ has either zero $\hat\mu_{1,S}$-measure for every $S$
or full $\hat\mu_{1,S}$-measure for every $S$.
This implies that there exists $c>0$ such that Radon-Nykodim derivative $d\hat\mu_{2,S}/d\hat\mu_{1,S}=c$ 
at $\hat\mu_{1,S}$-almost every point of every cross-section $S$. That implies the claim.
\end{proof}

Two invariant transverse measures $\tau_1=\{\hat\mu_{1,S}\}$ and $\tau_2=\{\hat\mu_{2,S}\}$ are said to be
\emph{mutually singular} if the measures $\hat\mu_{1,S}$ and $\hat\mu_{2,S}$ are mutually singular for every $S$.

\begin{lemma}\label{l.ergodictransverse2}
If $\tau$ is a non-ergodic invariant transverse measure then there exists a splitting $\tau = \tau_1  + \tau_2$
into mutually singular invariant transverse measures.
\end{lemma}

\begin{proof}
The assumption means that there is some splitting $\tau=\tau_1'+\tau_2'$ and neither $\tau_1'=\{\hat\mu'_{1,S}\}$
nor $\tau_2'=\{\hat\mu'_{2,S}\}$ are multiples of $\tau$. For each $S\in\cS$, let
$$
\hat\mu'_{2,S}=\hat\mu'_{2,S,s}+\hat\mu'_{2,S,a}
$$
be the Lebesgue decomposition of $\hat\mu'_{2,S}$ with respect to $\hat\mu'_{1,S}$, where $\hat\mu'_{2,S,s}$
denotes the singular part and $\hat\mu'_{2,S,a}$ denotes the absolutely continuous part.
There are two cases to consider. If $\hat\mu'_{2,S,s}\neq 0$ for some $S_0\in\cS$ then we may take
$\tau_1=\{\hat\mu_{1,S}\}$ and $\tau_2=\{\hat\mu_{2,S}\}$ with
$$
\hat\mu_{1,S}=\hat\mu'_{1,S}+\hat\mu_{2,S,a}
\text{ and } 
\hat\mu_{2,S}=\hat\mu'_{2,S,s}.
$$
Then $\tau_1$ and $\tau_2$ are (non-vanishing) invariant transverse measures with $\tau=\tau_1+\tau_2$,
and the construction immediately gives that they are mutually singular.

Now suppose that $\hat\mu'_{2,S,s} = 0$ or, in other words,  $\hat\mu'_{2,S}$ is absolutely continuous
with respect to $\hat\mu'_{1,S}$  for every $S\in\cS$.
For each cross-section $S$ and any interval $I\subset[0,\infty]$, denote by $I_S$ the set of
points where the Radon-Nykodim derivative satisfies $d\hat\mu'_{2,S}/d\hat\mu'_{1,S} \in I$.
Keep in mind that the Radon-Nykodim derivatives are invariant under holonomies, since the measures themselves
are, and so the family of sets $I_S$ is also invariant under holonomy.
The fact that $\tau_1'$ and $\tau_2'$ are not multiples of $\tau$ ensures that there exist $p\in(0,\infty)$
and a cross-section $S_0$ such that
$$
\hat\mu_{1,S_0}([0,p]_{S_0})>0 \text{ and } \hat\mu_{1,S_0}((p,\infty]_{S_0})>0.
$$
Let $\tau_1$ and $\tau_2$ be the restrictions of $\tau$ to $\{[0,p]_S: S \in \cS\}$ and
$\{(p,\infty]_S: S \in \cS\}$.
Once more, $\tau_1$ and $\tau_2$ are non-vanishing invariant transverse measures with $\tau=\tau_1+\tau_2$,
and they are mutually singular.
\end{proof}

\section{Proof of Theorem~\ref{main.A}}\label{s.proofA}

Let $\mu$ be any measure of maximal $u$-entropy of $f$. By Proposition~\ref{p.gibbs},
$\mu$ is a $c$-Gibbs $u$-state and its support is $u$-saturated.
Let $x_i$ and $x_j$ be two points in the support of $\mu$ contained in the same strong-unstable
leaf and in the interior of Markov partition elements $\cM_i$ and $\cM_j$.
See Figure~\ref{fig:holonomy1}.
Keep in mind that the boundaries of the Markov elements have zero $\mu$-measure,
by Proposition~\ref{p.gibbs}(5).
We denote by $\hat\mu_i$ the projection of $\mu \mid \cM_i$ to $\xi^{cs}_i(x_i)$ under
strong-unstable holonomy inside $\cM_i$, and similarly for $\hat\mu_j$.

\begin{figure}[ht]
\psfrag{xi}{\footnotesize{$x_i$}}\psfrag{xj}{\footnotesize{$x_j$}}
\psfrag{Mi}{$\cM_i$}\psfrag{Mj}{$\cM_j$}
\psfrag{Ei}{\footnotesize{$E$}}\psfrag{Ej}{\footnotesize{$H_{i,j}(E)$}}
\centering
\includegraphics[width=4in]{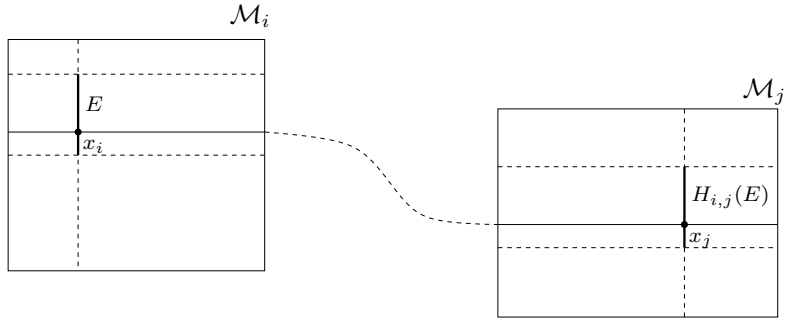}
\caption{\label{fig:holonomy1} Strong-unstable holonomies associated to
points $x_i$ and $x_j$, possibly far away from each other, on the same
strong-unstable leaf.}
\end{figure}

Let $H$ be the strong-unstable holonomy map from a small neighborhood of $x_i$ to a neighborhood
of $x_j$. We denote by $JH_{i,j}$ the Jacobian of $H$ with respect to the measures
$\hat\mu_i$ and $\hat\mu_j$. Let $\Leb^u$ denote (non-normalized) Lebesgue measure along unstable leaves.
For each $i=1, \dots, k$, define
$$
c_i = \frac{1}{\Leb^u(\pi\xi^u_i(x_i))}
$$
for any $x_i\in\cM_i$. This does not depend on the choice of the point $x_i$ because the volume
of the unstable plaques for the linear Anosov map $A$ is constant on each Markov rectangle $\cR_i$.

\begin{lemma}\label{l.invariance}
The Jacobian $JH_{i,j}$ is constant equal to ${c_i}/{c_j}$ for any $i$ and $j$.
\end{lemma}

\begin{proof}
Fix $N\ge 1$ large enough that $f^{-N}(x_i)$ and  $f^{-N}(x_j)$ are contained in the interior
of the same Markov partition element $\cM_l$. 
It is no restriction to assume that the domain $E\subset\xi_i^{cs}(x_i)$ of $H$ is small enough
that $f^{-n}(E)$ and $f^{-n}(H(E))$ are contained in the interior of $\cM_l$.
By the definition of $\hat\mu_i$ and $\hat\mu_j$, and the fact that $\mu$ is $f$-invariant,
$$
\hat\mu_i(E)
= \mu\left(\bigcup_{x\in E} \xi^u_i(x)\right)
= \mu\left(f^{-N} \left(\bigcup_{x\in E} \xi^u_i(x)\right)\right),
$$
and a similar fact holds for $\hat\mu_j\left(H(E)\right)$.
It is clear that the sets $f^{-N} \left(\bigcup_{x\in E} \xi^u_i(x)\right)$ and
$f^{-N} \left(\bigcup_{y \in H(E)} \xi^u_i(y)\right)$ project to the same set
$B\subset\xi^{cs}_l(x_l)$ under the strong-unstable holonomy inside $\cM_l$.
Then, since $\mu$ is a $c$-Gibbs $u$-state, Rokhlin disintegration gives that
\begin{equation}\label{eq.quotient1}
\hat\mu_i(E)
= \int_B \nu^u_{l,z}\left(f^{-N}(\xi^u_{i}(f^N(z)))\right) \, d\hat\mu_l(z) 
\end{equation}
and similarly for $\hat\mu_j\left(H(E)\right)$.
By the definition of the reference measures 
\begin{equation}\label{eq.quotient2}
\nu^u_{l,z}\left(f^{-N}(\xi^u_{i}(f^N(z)))\right)
= \frac{\Leb^u(\pi\left(f^{-N}\xi^u_{i}(f^N(z)))\right)}
{\Leb^u(\pi(\xi^u_l(z)))}.
\end{equation}
Since $f$ is semi-conjugate to the linear automorphism $A$, 
$$
\Leb^u(\pi\left(f^{-N}\xi^u_{i}(f^N(z)))\right)
= \frac{\Leb^u\left(\pi(\xi^u_i(z))\right)}{\det(A^{N} \mid E^u)} 
= \frac{\Leb^u\left(\pi(\xi^u_i(x_i))\right)}{\det(A^{N} \mid E^u)}
= \frac{1}{c_i \det(A^{N} \mid E^u)}.
$$
and, analogously,
$$
\Leb^u(\pi\left(f^{-N}\xi^u_{j}(f^N(w)))\right)
= \frac{1}{c_j \det(A^{N} \mid E^u)}.
$$
Replacing these last two identities in \eqref{eq.quotient1} and \eqref{eq.quotient2},
we find that
$$
\frac{\hat\mu_j(H(E))}{\hat\mu_i(E)}
= \frac{c_i}{c_j}.
$$
Since $E$ is arbitrarily small, this gives the claim of the lemma.
\end{proof}

Next, let $S$ be any cross-section to the strong-unstable foliation inside $\cM_i$.
It is no restriction to assume that $S$ meets every strong-unstable plaque $\xi_i^u(x)$.
Define $\hat\mu_S$ to be the projection of $\mu \mid \cM_i$ to $S$ under
strong-unstable holonomy inside $\cM_i$. It follows directly from Lemma~\ref{l.invariance}
that $\{c_i\hat\mu_S\}$ is an invariant transverse measure.
This proves Theorem~\ref{main.A}.

\section{Proof of Theorem~\ref{main.D}}\label{s.proofD}

In this section, we refer to the center leaf through each $x\in \cM_i$ as its \emph{center plaque},
and denote it as $\xi_i^c(x)$. Note that in the present context, all center leaves are compact, and each 
$\xi_i^c(x)$ is contained in the center-stable plaque $\xi_i^{cs}(x)$.
For each $i$, let $\mu_i$ denote the restriction $\mu \mid \cM_i$.

\begin{lemma}\label{l.finite1}
There is $N \ge 1$ and for each $i$ there exists a full $\mu_i$-measure subset of $\cM_i$ that intersects
each center plaque $\xi_i^c(x)$ on at most $N$ points.
\end{lemma}

\begin{proof}
Since $f$ is assumed to have $c$-mostly contracting center, the center-stable exponents of $\mu$ are all negative
(Proposition~\ref{p.contracting}), and so, there exist $m \ge 1$ and $c > 0$ such that
$$
\frac{1}{m}\int \log \|Df^{-m}\mid_{E^{cs}}\|^{-1} \, d\mu \ge c.
$$
This inequality remains true for some ergodic component $\mu_0$ of $\mu$ for the iterate $f^{m}$.
Then, by the Birkhoff ergodic theorem,
\begin{equation}\label{eq.m1}
\lim_n \frac{1}{n} \sum_{i=0}^{n-1} \log \|Df^{-m} \mid_{E^{cs}(f^{-i m}y)}\|^{-1} \ge c
\end{equation}
for $\mu_0$-almost every $y$. Let $K = \max\|Df^{\pm 1} \mid_{E^{cs}}\|$ and fix $l > 1$ large enough 
that $(l-1) c \ge 2 m \log K$. Since $\mu$ is assumed to be ergodic for $f$, it satisfies
$$
\mu = \frac{1}{m} \sum_{j=0}^{m-1} f^j_*\mu_0.
$$
So, for $\mu$-almost every $x$ there exists $j=0, \dots, m-1$ such that $y=f^j(x)$ satisfies \eqref{eq.m1},
and so
$$
\begin{aligned}
\lim_n \frac{1}{n} \sum_{i=0}^{n-1} & \log\|Df^{-lm}\mid_{E^{cs}(f^{-ilm}(x))}\|^{-1} \\
& \ge \lim_n \frac{1}{n} \sum_{i=0}^{n-1} \log \left(K^{-j} \|Df^{-lm} \mid_{E^{cs}(f^{-ilm}(y))}\|^{-1} K^{-j}\right)\\
& \ge - 2j \log K + l \lim_n \frac{1}{nl} \sum_{i=0}^{nl-1} \log \|Df^{-m} \mid_{E^{cs}(f^{-im}(y))}\|^{-1}\\
& \ge - 2 m \log K + l c \ge c.
\end{aligned}
$$
(in the second inequality we used the sub-multiplicativity of the norm).
By Proposition~3.7 in \cite{ViY13} it follows that there exists $N \ge 1$, depending only on $c$ and the maximum
volume of the center leaves, and there exists a full $\mu$-measure subset whose intersection with each center plaque
contains no more than $N$ points.
\end{proof}

Let $\{\mu^c_{i,x}: x \in \cM_i\}$ be a disintegration of $\mu \mid \cM_i$ along the center plaques $\xi^c_i(x)$.
It follows immediately from Lemma~\ref{l.finite1} that for $\mu_i$-almost every $x$ the conditional probability
$\mu^c_{i,x}$ is supported on no more than $N$ points.

Now let $\hat\mu_i$ be the quotient of the measure $\mu \mid \cM_i$ with respect to family of unstable
plaques $\xi^u_i(x)$. This may be viewed as a measure on the center-stable plaque $\xi_i^{cs}(x_i)$,
namely the projection of $\mu \mid \cM_i$ to $\xi_i^{cs}(x_i)$ along the unstable plaques.
Then, let $\{\hat\mu^c_{i,y}: y \in \xi^{cs}_i(x_i)\}$ be a disintegration of $\hat\mu_i$ along the
center plaques $\xi^c_i(y)$. Similarly to the previous paragraph, we have:  

\begin{lemma}\label{l.finite2}
For $\hat\mu_i$-almost every $y \in \xi_i^{cs}(x_i)$, the conditional probability $\hat\mu^c_{i,y}$ is supported on
no more than $N$ points.
\end{lemma}

\begin{proof}
By the local product structure, we may find coordinates $(x_u,x_{cs})$ on $\cM_i$ for which the unstable plaques 
$\xi^u_i$ and the center-stable plaques $\xi^{cs}_i$ are given by relations $x_{cs}=\constant$ and
$x^u=\constant$, respectively.
Since $\mu$ is a measure of maximal $u$-entropy and, thus, a $c$-Gibbs $u$-state, the restriction
$\mu \mid \cM_i$ disintegrates into the reference measures $\{\nu^u_{i,x}: x\in \cM_i\}$ along the
unstable plaques $\xi^u_i(x)$. Then
\begin{equation}\label{eq.disinto1}
\mu \mid \cM_i = \int_{\xi^{cs}_i(x_i)} \nu^u_{i,x_{cs}} \, d\hat\mu_i(x_{cs}).
\end{equation}
By \eqref{eq_cs-invariant}, the reference measures are invariant under center-stable holonomy.
In our coordinates, that means that $\nu^u_{i,x_{cs}}$ is constant, that is, independent of $x_{cs}$.
We write this constant as $\nu^u_{i}$, and view it as a measure on the unstable plaque $\xi^u_i(x_i)$.
Thus \eqref{eq.disinto1} becomes
\begin{equation}\label{eq.disinto2}
\mu \mid \cM_i = \nu^u_{i} \times \hat\mu_i
\end{equation}
Let $\hat\mu_i^s$ be the quotient measure of $\hat\mu_i$ with respect to the family of center plaques
$\xi^c_i(x_{cs})$, which we may view as a measure on the stable plaque $\xi^s_i(x_i)$. Then
\begin{equation}\label{eq.disinto3}
\hat\mu_i = \int_{\xi^s_i(x_i)} \hat\mu^c_{i,x_s} \, d\hat\mu^s_i(x_s),
\end{equation}
and so
\begin{equation}\label{eq.distinto4}
\mu \mid \cM_i 
= \nu^u_i \times \int_{\xi_i^s(x_i)}   \hat\mu^c_{i,x_s} \, d\hat\mu^s_i(x_s)
= \int_{\xi_i^u(x_i) \times \xi_i^s(x_i)}   \hat\mu^c_{i,x_s} \, d\hat\mu^s_i(x_s) \, d\nu_i^u(x_u).
\end{equation}
This shows that the quotient measure of $\mu \mid \cM_i$ with respect to the family of center plaques
is (in our coordinates) the measure $\nu^u_i \times \hat\mu^{s}_i$ on  $\xi^u_i(x_i) \times \xi^s_i(x_i)$.
Moreover, by the essential uniqueness of the disintegration, for $\mu$-almost every $x\in\cM_i$
\begin{equation}\label{eq.distinto5}
\mu^c_{i,x} = \hat\mu^c_{x_s}
\end{equation}
where $x_s$ is the point in the intersection of $\xi^{cu}_i(x)$ with $\xi^s(x_i)$.
Thus the claim follows directly from Lemma~\ref{l.finite1}.
\end{proof}

A foliation is said to be \emph{uniquely ergodic} if it admits a unique invariant transverse measure,
up to constant factor. The following lemma is a special case of Theorem~2.1 in Bowen, Marcus~\cite{BoM77}
which holds for the unstable foliation of any mixing basic set of an Axiom A diffeomorphism
(it is also not difficult to prove the lemma directly):

\begin{lemma}\label{l.uniquely_ergodic}
The unstable foliation of every Anosov linear automorphism $A$ is uniquely ergodic.
\end{lemma}



We are ready to complete the proof of Theorem~\ref{main.D}. It $\tau_\mu$ is ergodic, there is nothing to do.
Otherwise, by Lemma~\ref{l.ergodictransverse2} we may split is as $\tau_\mu=\tau_1+\tau_2$ where $\tau_1$ and
$\tau_2$ mutually singular. The projections of both under $\pi$ are invariant transverse measures for the
unstable foliation of the linear automorphism $A$. Thus, by Lemma~\ref{l.uniquely_ergodic}, they are multiples
of each other. For $j=1, 2$ and any $i$, let
\begin{equation}
\tau_j \mid \xi^{cs}_i(x_i) = \int \tau^c_{j,z} \, d\hat\tau_j(z)     
\end{equation}
be the disintegration of $\tau_j \mid \xi^{cs}_i(x_i)$ along center plaques $\xi_i^c(z)$.
By the previous remarks, $\hat\tau_1$ and $\hat\tau_2$ are multiples to each other.
So, the fact that $\tau_1 \mid \xi^{cs}_i(x_i)$ and $\tau_2 \mid \xi^{cs}_i(x_i)$ are mutually singular
implies that $\tau^c_{1,z}$ and $\tau^c_{2,z}$ are mutually singular for
$\hat\tau_j$-almost every $z$.
The sum $\tau^c_{1,z} + \tau^c_{2,z}$ is the conditional probability $\tau_{\mu,z}^c$ of the measure
$\tau_\mu \mid \xi^{cs}_i(x_i)$ which, by Theorem~\ref{main.A} is a multiple of $\hat\mu^c_i$.
By Lemma~\ref{l.finite2}, the latter is supported on no more than $N$ points. 
Thus the supports of the measures $\tau^c_{1,z}$ and $\tau^c_{2,z}$ must be disjoint, and this decomposition
procedure cannot be repeated more than $N$ times. This completes the proof.

\section{Proof of Theorem~\ref{main.B}}\label{s.proofB}

Let $x_1$ and $x_2$ be any two points in $\supp\mu$.
Up to renumbering the elements of the Markov partition if necessary, we may suppose that $x_1 \in \cM_1$ and $x_2 \in \cM_2$.
For $i=1, 2$, denote $Y_i=\xi^u_i(x_i) \times [0,1]$. Equip each $Y_i$ with the probability measure $m_i=\nu^u_{i,x_i} \times dt$.

The main technical step in the proof of Theorem~\ref{main.B} is the following lemma,
whose proof we postpone to Section~\ref{s.coupling2}.
This is also the one step where we use the assumption that the support of $\mu$ is connected.

\begin{lemma}[Coupling Lemma]\label{l.coupling}
There are a map $\tau: Y_1\to Y_2$ with $\tau_*m_1 = m_2$, a function $R: Y_1 \to \NN$,
and constants $C_1, C_2>0$ and $\rho_1, \rho_2 \in(0,1)$
such that
\begin{enumerate}
\item If $\tau(x,t)=(y,s)$ then $f^n(x)$ and $f^n(y)$ belong to the same Markov component $\cM_i$ for some $i$,
and $d(f^n(x),f^n(y))\leq C_1\rho_1^{n-R}$
for $n\geq R(x,t)$.
\item $m_1(R>n)\leq C_2\rho_2^n$ for every $n\ge 1$.
\end{enumerate}
\end{lemma}

Let $S$ be a cross-section contained in some $\cM_j$ with $\mu(\cM_j)>0$,
and $\hat\varphi:S \to \RR$ be any H\"older real function supported inside $S$.
In the remainder of this section, we take $\varphi:M \to \RR$ to be the extension of $\hat\varphi$ which is
constant along unstable plaques on $\cM_j$ and vanishes on $M \setminus \cM_j$.
Since  $\hat\varphi$ is taken to be supported inside $S$,
the extension $\varphi$ vanishes on the unstable boundary $\partial^u \cM_j = \pi^{-1}(\partial^u\cR_j)$. 
Notice also that $\varphi$ is H\"older restricted to $\cM_j$, since $\hat\varphi$ is H\"older and so are 
the strong-unstable holonomy maps.

\begin{corollary}\label{c.coupling1}
Given $\hat\varphi:S\to\RR$ there are $C>0$ and $0<\rho<1$ such that for any points $x,y\in \supp\mu$ 
$$
\left|\int \varphi \, d\left(f^n_* \nu^u_{i,x}\right) -\int \varphi \, d\left(f^n_* \nu^u_{k,y}\right)\right| \leq C\rho^n
$$
for every $n\ge 1$, where $\cM_i \ni x$ and $\cM_k\ni y$. 
\end{corollary}

\begin{proof}
It is no restriction to suppose that $i=1$ and $k=2$.
Define $\tilde\varphi:M \times [0,1] \to \RR$ by  $\tilde\varphi(x,t) = \varphi(x)$. Then
\begin{equation}\label{eq.integ1}
\int \varphi \, d\left(f^n_* \nu^u_{1,x}\right)
= \int (\varphi \circ f^n) \,  d\nu^u_{1,x}
= \int \tilde\varphi (f^n(z),t) \, d m_1(z,t).
\end{equation}
By Lemma~\ref{l.coupling}, the map $\tau: \xi^u_{1,x}\times [0,1] \to \xi^u_{2,y}\times [0,1]$,
$(z,t) \mapsto (w,s) = \tau((z,t))$ sends the measure $m_1$ to $m_2$. Thus,
\begin{equation}\label{eq.integ2}
\int \varphi \, d\left(f^n_* \nu^u_{2,x}\right)
= \int \tilde\varphi (f^n(w),s) \, d m_2(w,s)
= \int \tilde\varphi (f^n(w),s) \, d m_1(z,t).
\end{equation}
We are going to break both \eqref{eq.integ1} and \eqref{eq.integ2} as a sum of integrals over the 
two domains $\{R(z,t) \le n/2\}$ or $\{R(z,t) > n/2\}$.
Let $C_3$ and $\alpha_3$ be H\"older constants for $\varphi$ restricted to $\cM_j$.
By part (1) of Lemma~\ref{l.coupling},
$$
\begin{aligned}
& \left|\int_{R(z,t) \le n/2} \Big[\tilde\varphi (f^n(z),t) - \tilde\varphi (f^n(w),s) \Big] \, d m_1(z,t)\right| \\
& \qquad \qquad = \left|\int_{R(z,t) < n/2} \Big[\varphi (f^n(z)) - \varphi (f^n(w)) \Big] \, d m_1(z,t)\right|\\
& \qquad \qquad \le \left|\int_{R(z,t) < n/2} C_3 \left(C_1 \rho_1^{n-R(z,t)}\right)^{\alpha_3} \, d m_1(z,t)\right|
\le C_3 C_1^{\alpha_3} \rho_1^{\alpha_3 n/2}.
\end{aligned}
$$
By part (2) of Lemma~\ref{l.coupling}, the integrals over $\{R(z,t) > n/2\}$ are bounded above by $C_2\rho_2^{n/2}\sup|\varphi|$.
The claim is a direct consequence of these two estimates, with 
$C = C_3C_1^{\alpha_3} + C_2\sup|\varphi|$ and $\rho = \max\{\rho_1^{\alpha_3/2},\rho_2^{1/2}\}$.
\end{proof}

\begin{lemma}\label{l.coupling2}
There is $n_0>0$ such that for any $n \ge n_0$, any $x\in \cM_i\cap \supp \mu$, and any $i$
$$
\int \varphi \, d\left(f^n_* \nu^u_{i,x}\right) = \left(f^n_*\nu^u_{i,x}\right)(\cM_j) \frac{1}{\# \big(f^n(\xi^u_{i}(x)) \cap S\big)}\sum_{q \in f^n(\xi^u_{i}(x))\cap S} \hat\varphi(q).
$$
\end{lemma}

\begin{proof}
By the Markov property, the intersection of $f^n(\xi^u_{i}(x))$ with $\cM_j$ consists of the unstable plaques
$\xi^u_{j}(q)$, $q \in f^n(\xi^u_{i}(x)) \cap S$. 
The assumption that $f$ is semi-conjugate to a linear Anosov map $A$ ensures that there exists $n_0\ge 1$
independent of $x$ such that this intersection is non-empty for every $n \ge n_0$.
Since $\varphi$ is constant on unstable plaques inside $\cM_j$,
$$
\int \varphi \, d\left(f^n_* \nu^u_{i,x}\right)
= \sum_{q\in f^n(\xi^u_{i}(x)) \cap S} \nu^u_{i,x}(f^{-n}(\xi^u_{j}(q))) \hat\varphi(q).
$$
By the definition of the reference measures, each term 
$\nu^u_{i,x}(f^{-n}(\xi^u_{j}(q))$ coincides with the Lebesgue measure of the image of 
$f^{-n}(\xi^u_{j}(q))$ under the semi-conjugacy, which is the $A^{-n}$-image of an unstable plaque $\cR_j$.
Since all these unstable plaques have the same Lebesgue measure, and the Jacobian of $A^{-n}$ along the
unstable direction is constant everywhere, we find that $\nu^u_{i,x}(f^{-n}(\xi^u_{j}(q)))$ is independent
of $q$. Therefore,
$$
\begin{aligned}
\nu^u_{i,x}(f^{-n}(\xi^u_{j}(q))
& = \frac{1}{\# \big(f^n(\xi^u_{i}(x)) \cap S\big)} \sum_{w \in f^n(\xi^u_{i}(x))\cap S} \nu^u_{i,x}
\left(f^{-n}(\xi^u_{j}(w)) \right) \\
& = \frac{1}{\# \big(f^n(\xi^u_{i}(x)) \cap S)} \left(f^n_*\nu^u_{i,x}\right)(\cM_j)
\end{aligned}
$$
The claim follows directly from these two identities.
\end{proof}

We are ready to complete the proof of Theorem~\ref{main.B}.
By assumption, $\mu$ is a measure of maximal $u$-entropy, and so it is a $c$-Gibbs $u$-state.
This means that its conditional measures along unstable plaques $\xi^u_i(x)$ are the reference measures $\nu^u_{i,x}$.
Thus, recalling also that $\mu$ is $f$-invariant, Corollary~\ref{c.coupling1} implies that 
$$
\left|\int \varphi \, d\left(f^n_* \nu^u_{i,x}\right) -\int \varphi \, d\mu\right| \leq C\rho^n
$$
for every $n\ge 1$. Then, using Lemma~\ref{l.coupling2}, 
\begin{equation}\label{eq.converge1}
\left|\left(f^n_*\nu^u_{i,x}\right)(\cM_j) \frac{1}{\# \big(f^n(\xi^u_{i}(x)) \cap S\big)}\sum_{q \in f^n(\xi^u_{i}(x))\cap S} \hat\varphi(q) -\int \varphi \, d\mu\right| \leq C\rho^n
\end{equation}
for every $n \ge n_0$. By the definition of the reference measures,
$$
\left(f^n_*\nu^u_{i,x}\right)(\cM_j)
= \left(A^n_*\Leb^u_{i,\pi(x)}\right)(\cR_j)
$$
since the projection $\pi:M\to\TT^d$ is a homeomorphism on each strong-unstable leaf.
It is a classical fact about linear Anosov maps that the right hand side converges to $\Leb(\cR_j)$ exponentially fast,
where $\Leb$ denotes the Haar measure on the torus.
This can be traced back to Sinai~\cite{Sin72} and Bowen~\cite{Bow75a},
see also Dolgopyat~\cite{Dol00} for a much more general statement.
Then, since $\mu$ is a $c$-Gibbs $u$-state, we can use \cite[Corollary~3.5]{UVYY1} to conclude that 
$$
\Leb(\cR_j)
= \mu(\cM_j) 
= \|\hat\mu_S\|.
$$
Keep also in mind that $\int \varphi \, d\mu = \int \hat\varphi \, d\hat\mu_S$. 
Thus, \eqref{eq.converge1} implies that
$$
\frac{1}{\# \big(f^n(\xi^u_i(x) \cap S\big)} \sum_{q \in f^n(\xi^u_i(x)) \cap S} \hat\varphi(q) 
\to \frac{1}{\|\hat\mu_S\|} \int \hat\varphi \, d\hat\mu_S
$$
exponentially fast as $n\to\infty$.

At this point we have reduced the proof of Theorem~\ref{main.B} to proving Lemma~\ref{l.coupling}.

\section{Large deviations}\label{s.large_deviations}

In this section we prove the following result which is need for the proof of the Coupling Lemma~\ref{l.coupling}:

\begin{theorem}\label{t.large_deviations}
Let $f:M\to M$ be a partially hyperbolic diffeomorphism that factors over Anosov and has $c$-mostly contracting center.
Let $\mu$ be a measure of maximal $u$-entropy whose  support $\Lambda$ is connected.
Then $(f,\mu)$ satisfies a large deviations principle for continuous observables.
\end{theorem}

One says that a system $(f,\mu)$ satisfies a \emph{large deviations principle} for continuous observables
if for every $\alpha>0$ and every continuous function $\varphi$ with
$\int_M \varphi \, d\mu = 0$ there exist positive constants $C_\alpha$ and $c_\alpha$
such that
\begin{equation}\label{eq.large_deviations_def}
\mu\left(\left\{x\in M: \left|\frac 1n \sum_{j=0}^{n-1} \varphi(f^j(x))\right| >\alpha\right\}\right))
\le C_\alpha e^{-c_\alpha n} \text{ for every }n\geq 1.
\end{equation}

Throughout the proof of Theorem~\ref{t.large_deviations}, which occupies the remainder of this section,
it is assumed that $\mu$ is a measure of maximal $u$-entropy and thus (compare part (4) of Proposition~\ref{p.gibbs})
a $c$-Gibbs $u$-state. The assumption that the support is connected is not essential, as explained in 
Remark~\ref{r.not_connected}.

\subsection{Probability measures with H\"older densities}\label{ss.Holder_densities}

Let us fix some element $\cM_i$ of the Markov partition (recall Section~\ref{s.Markov.partitions}).
It is convenient to consider a certain family of spaces
$E_i(R)$ consisting of probability measures on $\cM_i\cap\Lambda$ whose conditional probabilities
along the partition $\xi^u_i$ are absolutely continuous with respect to the reference measures,
with H\"older densities. We begin by giving the definition of this space, which is mostly
borrowed from Dolgopyat~\cite[Section~5]{Dol00}.

Let $\gamma\in(0,1)$ be fixed. For each $R \ge 0$ and $i=1, \dots, k$ denote by $C_i(R)$ the set
of all probability measures $\eta$ on $\cM_i\cap\Lambda$ of the form
\begin{equation}\label{eq.rho_dis}
\eta = e^{\rho} \nu^u_{i,x},
\end{equation}
where $x\in\cM_i\cap\Lambda$ and $\rho:\xi^u_i(x) \to \RR$ is an $(R,\gamma)$-H\"older function:
\begin{equation}\label{eq.rho_Holder}
|\rho(z_1)-\rho(z_2)| \leq R d(z_1,z_2)^{\gamma}
\text{ for any } z_1,z_2\in \xi^u_i(x).
\end{equation}
For instance, $C_i(0)$ consists precisely of the reference measures $\nu^u_{i,x}$, $x\in\cM_i\cap\Lambda$.
Every $C_i(R)$ is a weak$^*$-closed subset of the space of all probability measures on $\cM_i\cap\Lambda$,
because the reference measures $\nu^u_{i,x}$ vary continuously with $x$, and the space of
$(R,\gamma)$-H\"older functions is equicontinuous. Thus $C_i(R)$ is compact for the weak$^*$
topology, for every $R \ge 0$.

Next, let $\hat{E}_i(R)$ be the space of all probability measures on the compact space $C_i(R)$.
Note that $\hat{E}_i(R)$ is compact for the weak$^*$ topology. Consider the map
$$
\Pi: \hat{E}_i(R) \to \{\text{probability measures on } \cM_i\cap\Lambda\},
\quad
\Pi(\hat\zeta) = \int_{C_i(R)} \eta \, d\hat\zeta(\eta).
$$
In other words, $\Pi(\hat\zeta)$ is the probability measure on $\cM_i\cap\Lambda$ such that
$$
\int_{\cM_i\cap\Lambda} \psi \, d\Pi(\hat\zeta)
= \int_{C_i(R)} \left(\int_{\cM_i\cap\Lambda} \psi \, d\eta \right) \, d\hat\zeta(\eta)
$$
for any bounded measurable function $\psi:\cM_i\cap\Lambda\to\RR$.
Take $\psi$ to be continuous. Then it is clear that
\begin{equation}\label{eq.corresponding}
\hat\psi: C_i(R) \to \RR, \quad \hat\psi(\eta) = \int_{\cM_i\cap\Lambda} \psi \, d\eta
\end{equation}
is continuous. Let $(\hat\zeta_j)_j$ be any sequence converging to some $\hat\zeta$ in $\hat{E}_i(R)$.
Then
$$
\int_{\cM_i\cap\Lambda} \psi \, d\Pi(\hat\zeta_j) = \int_{C_i(R)} \hat\psi \, d\hat\zeta_j
\to
\int_{C_i(R)} \hat\psi \, d\hat\zeta = \int_{\cM_i\cap\Lambda} \psi \, d\Pi(\hat\zeta)
$$
as $j\to\infty$. Since $\psi$ is an arbitrary continuous function,
this proves that the map $\Pi$ is continuous.

Let $E_i(R) = \Pi(\hat{E}_i(R))$. It follows from the previous observations that $E_i(R)$ is a
weak$^*$-compact subset of the space of all probability measures on $\cM_i\cap\Lambda$.

\begin{lemma}\label{l.E_disintegration}
A probability measure $\zeta$ on $\cM_i\cap\Lambda$ is in $E_i(R)$ if and only if its conditional measures
with respect to the partition $\xi_i^u$ are elements of $C_i(R)$, that is, probability measures
of the form \eqref{eq.rho_dis}.
\end{lemma}

\begin{proof}
Suppose that $\zeta\in E_i(R)$. Then, by definition, there exists $\hat\zeta \in C_i(R)$ such that
$\zeta = \Pi(\hat\zeta)$. Consider the canonical map
$$
H: C_i(R) \to \xi_i^u,
\quad H\left(e^{\rho} \nu^u_{i,x}\right) \to \xi^u_i(x),
$$
and let $\tilde\zeta=H_*\hat\zeta$.
The partition of $\xi_i^u$ into points is measurable, because $\xi_i^u$ itself is a measurable
partition of $\cM_i\cap\Lambda$, and so, its pull-back under $H$ is a measurable partition of $C_i(R)$.
Let $\{\hat\zeta_P: P \in \xi_i^u\}$ be the disintegration of $\hat\zeta$ with respect to the pull-back:
each $\hat\zeta_P$ is a probability measure on $C_i(R)$ with $\hat\zeta_P\left(H^{-1}(P)\right)=1$,
and these probabilities satisfy
$$
\int_{C_i(R)} \hat\psi \, d\hat\zeta
= \int_{\xi_i^u} \left(\int_{H^{-1}(P)} \hat\psi \, d\hat\zeta_P\right) d\tilde\zeta(P)
$$
for every bounded measurable function $\hat\psi:C_i(R) \to \RR$.
Let $\psi:\cM_i\cap\Lambda\to\RR$ be any bounded measurable function, and $\hat\psi$ be as in \eqref{eq.corresponding}. Then,
$$
\int_{\cM_i\cap\Lambda} \psi \, d\zeta
= \int_{\cM_i\cap\Lambda} \psi \, d\Pi\hat\zeta
= \int_{C_i(R)} \hat\psi \, d\hat\zeta
= \int_{\xi_i^u} \left(\int_{H^{-1}(P)} \hat\psi \, d\hat\zeta_P\right) d\tilde\zeta(P).
$$
Moreover, by the definition \eqref{eq.corresponding},
$$
\begin{aligned}
\int_{H^{-1}(P)} \hat\psi(\eta) \, d\hat\zeta_P(\eta)
& = \int_{H^{-1}(P)} \int_{\cM_i\cap\Lambda} \psi \, d\eta \, d\hat\zeta_P(\eta)\\
& = \int_{\cM_i\cap\Lambda} \psi \, d\left(\int_{H^{-1}(P)} \eta \, d\hat\zeta_P(\eta)\right).
\end{aligned}
$$
This means that the conditional probabilities of $\zeta$ with respect to the partition $\xi^u_i$
are the measures
$$
\int_{H^{-1}(P)} \eta \, d\hat\zeta_P(\eta), \quad P \in \xi^u_i.
$$
Consider any $P=\xi^u_i(x)$. The condition $\eta\in H^{-1}(P)$ means that $\eta$ is a measure
of the form $e^\rho\nu_{i,x}^u$. Then
$$
\int_{H^{-1}(P)} \eta \, d\hat\zeta_P(\eta) = \left(\int e^\rho \, d\zeta_{i,x}(\rho)\right) \nu^u_{i,x},
$$
where $\zeta_{i,x}$ is a probability measure on the space of $(R,\gamma)$-H\"older functions.
To conclude it suffices to note that the function
$$
y \mapsto \log \left(\int e^{\rho(y)} \, d\zeta_{i,x}(\rho)\right)
$$
is $(R,\gamma)$-H\"older. Indeed, given any $y_1, y_2 \in\xi^u_i(x)$,
$$
\log \frac{\int e^{\rho(y_1)} \, d\zeta_{i,x}(\rho)}
{\int e^{\rho(y_2)} \, d\zeta_{i,x}(\rho)}
\le \log \frac{\int e^{\rho(y_2)+Rd(y_1,y_2)^\gamma} \, d\zeta_{i,x}(\rho)}
{\int e^{\rho(y_2)} \, d\zeta_{i,x}(\rho)}
= Rd(y_1,y_2)^\gamma.
$$
This proves the 'only if' part of the statement.

Now suppose that the disintegration $\zeta = \int_{\xi^u_i} \zeta_P \, d\tilde\zeta(P)$
of $\zeta$ with respect to $\xi_i^u$ is such that $\zeta_P \in C_i(R)$ for every $P\in\xi_i^u$.
Consider the measurable map $P \mapsto \zeta_P$ from $\xi_i^u$ to $C_i(R)$,
and let $\hat\zeta$ be the push-forward of $\tilde\zeta$ under this map.
Then $\hat\zeta$ is a probability measure on $C_i(R)$, that is, an element of $\hat{E}_i(R)$,
and
$$
\zeta = \int_{C_i(R)} \eta \, d\hat\zeta.
$$
In other words, $\zeta = \Pi(\hat\zeta)$, which proves that $\zeta\in E_i(R)$.
\end{proof}

Finally, define $C(R)$ to be the disjoint union of all $C_i(R)$, $i=1, \dots, k$,
and $\hat{E}(R)$ to be the space of probability measures on $C(R)$.
Clearly, the latter coincides with the space of convex combinations
$$
\hat\zeta = \sum a_i\hat\zeta_i
\text{ with } \sum_{i=1}^k a_i=1 \text{ and } \hat\zeta_i \in \hat{E}_i(R) \text{ for all }  i.
$$
Finally, define $E(R)$ to be the space of convex combinations
$$
\zeta = \sum_{i=1}^{k} a_i \zeta_i
\text{ with } \sum_{i=1}^k a_i =1
\text{ and } \zeta_i \in E_i(R) \text{ for all } i.
$$
It is clear that $C(R)$, $\hat{E}(R)$ and $E(R)$ are weak$^*$-compact spaces.

\begin{remark}\label{r.continuity}
The families $C_i(R), C(R), \hat{E}_i(R), \hat{E}(R), E_i(R), E(R)$ are clearly monotone increasing in $R$.
Moreover, they are continuous at $R=0$.
Indeed, it is clear that $C_i(0)$ coincides with the intersection of the compact sets $C_i(R)$ over
all $R>0$. It follows that $\hat{E}_i(0)$ coincides with the intersection of the $\hat{E}_i(R)$
over all $R>0$. Since $R\mapsto \hat{E}_i(R)$ is a monotone family of compact sets,
and $\Pi$ is continuous, we also get that
$$
E_i(0)
= \Pi(\hat{E}_i(0))
= \Pi\left(\bigcap_{R>0} \hat{E}_i(R)\right)
= \bigcap_{R>0} \Pi\left( \hat{E}_i(R)\right)
= \bigcap_{R>0} E_i(R).
$$
Since this is true for every $i$, the corresponding statements for $C(R)$, $\hat{E}(R)$ and $E(R)$
follow.
\end{remark}

Let $1/\omega>1$ be a lower bound for the expansion rate of $f$ along strong-unstable leaves as
in \eqref{eq.omega}.

\begin{proposition}\label{p.Econtracting}
$f_*\left(E(R)\right) \subset E(Re^{l\gamma\log \omega})$ for any $R\ge 0$.
\end{proposition}

\begin{proof}
Let us begin by considering $\eta=e^{\rho}\nu^u_{i,x}$ in any $C_i(R)$.
The push-forward of the reference measure $\nu^u_{i,x}$ is a finite convex combination of reference
measures $\nu^{j,y_j}$:
$$
f_*\nu^u_{i,x} = \sum_{j\in J(i)} a_j \nu_{j,y_j}^u.
$$
See \eqref{eq_constant_Jacobian} and Remark~\ref{r.lowboundary}.
Thus the push-forward of $\eta$ may be written as
$$
f_* \eta = \sum_{j\in J(i)} \left(a_je^{-b_j}\right) e^{b_j + \rho \circ f^{-1}}\nu^u_{j,y_j},
$$
where the exponents $b_j$ are chosen so that each $e^{b_j + \rho \circ f^{-1}}\nu^{j,y_j}$ is
a probability measure. By assumption, $\rho$ is $(R,\gamma)$-H\"older on the strong-unstable
plaque $\xi^u_i(x)$.
Since $f^{-1}$ contracts strong-unstable leaves at a uniform rate $e^{l\log \omega}$,
it follows that $b_j + \rho \circ f^{-1}$ is $(Re^{l\gamma\log\omega},\gamma)$-H\"older.
Thus, the previous identity may be written as
$$
f_*\eta = \int_{C(R e^{l\gamma\log\omega})} \zeta \, d\hat\zeta_\eta(\zeta),
$$
where $\hat\zeta_\eta$ is the element of $\hat{E}(Re^{l\gamma\log\omega})$ given by the convex
combination of Dirac masses at the $e^{b_j + \rho \circ f^{-1}}\nu^u_{j,y_j}$,
with the $a_je^{-b_j}$ as the coefficients.

Now let $\zeta$ be any element of $E(R)$. Then there exists $\hat\zeta\in\hat{E}(R)$ such that
$$
\zeta = \Pi(\hat\zeta) = \int_{C(R)} \eta \, d\hat\zeta(\eta).
$$
Then
$$
\begin{aligned}
f_* \zeta
& = \int_{C(R)} f_*\eta \, d\hat\zeta(\eta)
= \int_{C(R)} \int_{C(R e^{l\gamma\log\omega})} \zeta \, d\hat\zeta_\eta(\zeta) \, d\hat\zeta(\eta)\\
& = \int_{C(R e^{l\gamma\log\omega})} \zeta \, d\left(\int_{C(R)} \hat\zeta_\eta \, d\hat\zeta(\eta)\right)(\zeta).
\end{aligned}
$$
To conclude, observe that
$$
\int_{C(R)} \hat\zeta_\eta \, d\hat\zeta(\eta) \in \hat{E}(R e^{l\gamma\log\omega})
$$
because the latter is a convex compact space and it contains every $\hat\zeta_\eta$.
Thus, the previous identity means that $f_*\zeta \in E(R e^{l\gamma\log\omega})$.
\end{proof}

\begin{lemma}\label{l.Eunique}
For every $n \ge 1$, $\mu$ is $f^n$-ergodic and it is the unique $f^n$-invariant probability measure in $E(0)$.
\end{lemma}

\begin{proof}
By \cite[Theorem~A]{UVYY1}, $\mu$ has finitely many ergodic components for $f^n$,
and their supports are pairwise disjoint. Since $\Lambda=\supp\mu$ is connected,
it follows that there can be only one ergodic component, in other words,
$\mu$ is $f^n$-ergodic.

Since $\mu$ is a $c$-Gibbs $u$-state, it follows directly from Lemma~\ref{l.E_disintegration} that $\mu\in E(0)$.
To prove uniqueness, let $\tilde\mu$ be any $f^n$-invariant probability measure in $E(0)$.
By definition, $\tilde\mu$ is supported in $\Lambda$ and can be written as a convex combination
$$
\tilde\mu = \sum_{i=1}^k a_i \tilde\mu_i \text{ with } \tilde\mu_i \in E_i(0).
$$
We claim that $\tilde\mu(\partial\cM_i)=0$ for every $i=1, \dots, k$.
Thus the restriction of $\tilde\mu$ to each $\cM_i$ coincides with $a_i\tilde\mu_i$.
Now, by Lemma~\ref{l.E_disintegration} the conditional probabilities of each $\tilde\mu_i$ along the plaques $\xi^u_i$
are the reference measures. Thus, $\tilde\mu$ is a $c$-Gibbs $u$-state. Using \cite[Corollary~4.6]{UVYY1},
we get that $\tilde\mu=\mu$, as we wanted to prove.

To prove our claim we only need to check that the projection $\tilde\nu=\pi_* \tilde\mu$ is the Lebesgue measure on $\TT^d$.
Note that $\tilde\nu$ is an invariant probability measure for the linear Anosov map $A$,
and
$$
\tilde\nu = \sum_{i=1}^k a_i \tilde\nu_i
$$
where each $\tilde\nu_i=\pi_*\tilde\mu_i$ is a probability measure on the Markov set $\cR_i$.
By the definition of the reference measures, the conditional probability of $\tilde\nu_i$
along every unstable plaque $W^u_{i,a}$ is the normalized Lebesgue
measure.

Consider any $i$ such that $a_i>0$, and then let $\{\tilde\nu_a: a \in\cR_i\}$ denote
the conditional probabilities of $\tilde\nu \mid \cR_i$ along the
unstable plaques $W^u_{i,a}$.
Since the Lebesgue measure $\Leb$ is invariant and ergodic for $A$, the set of points
$b \in \TT^d$ such that
\begin{equation}\label{eq.goesto}
\frac{1}{n} \sum_{j=0}^{n-1} \delta_{A^j(b)} \to \Leb
\end{equation}
has full Lebesgue measure, and is $s$-saturated. Then, using also the fact that the
stable foliation of $A$ is absolutely continuous (linear, actually) we get that
\eqref{eq.goesto} holds for $\Leb_{i,a}$-almost every $b \in W^u_{i,a}$ and every
$a \in \cR_i$. Consequently, \eqref{eq.goesto} holds for a full $\tilde\nu_i$-subset
of points in $\cR_i$. Thus, it holds at $\tilde\nu$-almost every point, which
implies that $\tilde\nu=\Leb$, as claimed.
\end{proof}

\subsection{Proof of Theorem~\ref{t.large_deviations}}

Denote $S_n\varphi = \sum_{j=0}^{n-1} \varphi \circ f^j$ for any $n\ge 1$ and any continuous function
$\varphi:M\to\RR$. The first step in the proof of the theorem is:

\begin{lemma}\label{l.largesteplog}
For any $\alpha>0$ and any continuous function $\varphi:M\to\RR$ with
$\int \varphi \, d\mu \leq -\alpha$, there is $C_1>0$ such that
$$
\int_{\xi^u_i(x)} S_n\varphi \, d\nu^u_{i,x}
\leq -n\frac{\alpha}{2} + C_1
\text{ for any $n \ge 1$, $x \in \cM_i \cap \Lambda$, and $i=1, \dots, k$,}
$$
\end{lemma}

\begin{proof}
Recall that $E(0)$ is convex, weak$^*$-compact, and $f_*$-invariant, and $\mu$ is the unique $f$-invariant probability measure
contained in it (Lemma~\ref{l.Eunique}). It follows that
$$
\frac{1}{n} \sum_{j=0}^{n-1} f_*^j \zeta \to \mu
$$
for any probability measure $\zeta\in E(0)$, and the convergence is uniform in $\zeta$.
Indeed, suppose that there exist $(\zeta_k)_k$ in $E(0)$ and $(n_k)_k \to \infty$ such that
\begin{equation}\label{eq.unique_ergodicity}
\frac{1}{n_k} \sum_{j=0}^{n_k-1} f_*^j \zeta_k
\end{equation}
remains far from $\mu$ for all $k$. Up to restricting to a subsequence, we may assume that \eqref{eq.unique_ergodicity}
converges to some probability measure $\tilde\mu$. It is clear that $\tilde\mu\in E(0)$ and $\tilde\mu \neq \mu$.
It is also well known that $\tilde\mu$ is necessarily $f$-invariant. This would contradict Lemma~\ref{l.Eunique}.

In particular,
$$
\frac{1}{n} \int_{\xi^u_i(x)} S_n \varphi \, d\nu^u_{i,x}
\to \int_{\Lambda} \varphi \, d\mu
$$
uniformly in $x$ and $i$. This implies the claim of the lemma.
\end{proof}

\begin{lemma}\label{l.distortion}
For any $\epsilon>0$ and any continuous function $\varphi:M\to\RR$,
there exist $C>0$ and $n_\epsilon \ge 1$ such that
$$
|S_n\varphi(y)-S_n\varphi(z)| \leq n \epsilon + C
$$
for any $n \geq n_\epsilon$, $y, z \in f^{-n}(\xi^u_i(x))$,
$x\in\cM_i\cap\Lambda$, and $i=1, \dots, k$.
\end{lemma}

\begin{proof}
This is because $\varphi$ is uniformly continuous, the diameter of $\xi^u_i(x)$ is uniformly
bounded, and $f^{-1}$ contracts strong-unstable leaves uniformly.
\end{proof}

The Markov property \eqref{eq.Markov1} implies that every $f^n(\xi_i^u(x))$
may be written as a (finite) union of strong-unstable plaques $\xi^u_{i_j}(x_j)$. Let
$$
c_j=c_j(i,x,n) = \nu^u_{i,x}(f^{-n}(\xi^u_{i_j}(x_j))),
$$
and note that $\sum_j c_j = 1$. Recall also (Remark~\ref{r.lowboundary}) that,
for any fixed $n$, there are only finitely many possible values for $c_j(i,x,n)$.

\begin{corollary}\label{c.discrete}
For any $\alpha>0$ and any continuous function $\varphi:M \to\RR$ with $\int \varphi \, d\mu \leq -\alpha$,
there are $\alpha_1>0$ and $n_1 \ge 1$ such that
$$
\sum_j c_j \max_{f^{-n}(\xi^u_{i_j}(x_j))} S_n\varphi \leq - n \alpha_1,
$$
for every $n\ge n_1$, $x\in\cM_i\cap \Lambda$, and $i=1, \dots, k$.
\end{corollary}

\begin{proof}
By Lemma~\ref{l.largesteplog} the average of $S_n\varphi$ is bounded above by
$-n\alpha/2+C_1$. By Lemma~\ref{l.distortion}, its total oscillation is bounded
by $n\epsilon+C$. Fix $\epsilon = \alpha_1 = \alpha/5$ and then take $n_1>n_\epsilon$
large enough that $(-n\alpha/2 + C_1) + (n\epsilon +C) \leq - n \alpha_1$
Thus,
$$
\sum_j c_j \max_{f^{-n}(\xi^u_{i_j}(x_j))} S_n\varphi \leq - n \alpha_1,
$$
which clearly implies the claim.
\end{proof}

\begin{lemma}\label{l.iterate_sufficient}
If Theorem~\ref{t.large_deviations} holds for some iterate $f^l$, $l\ge 1$ then it holds for $f$.
\end{lemma}

\begin{proof}
Start by noting that the assumptions of the theorem hold for $f^l$ if (and only if) they hold
for $f$. Indeed, it is clear from the definition \eqref{eq.c_mostly_contracting} that $f$
has $c$-mostly contracting center if and only if $f^l$ has $c$-mostly contracting center.
Similarly, $\mu$ is a $c$-Gibbs $u$-state for $f$ if and only if it is a $c$-Gibbs $u$-state
for $f^l$. Thus (by part (4) of Proposition~\ref{p.gibbs}), the maps $f$ and $f^l$ have
precisely the same measures of maximal $u$-entropy. Moreover, by Lemma~\ref{l.Eunique} one
has ergodicity with respect to any of the maps $f$ and $f^l$.

Now, we check that the conclusion of the theorem holds for $f$ if it holds for $f^l$.
Indeed, given any continuous function $\varphi:M \to \RR$, denote 
$$
S_n\varphi = \sum_{j=0}^{n-1} \varphi \circ f^j,\
\Sigma_n\varphi = \sum_{j=0}^{n-1} \varphi \circ f^{lj},
\text{ and }
\Phi=\sum_{j=0}^{l-1} \varphi \circ f^j.
$$
Note that $S_{nl}\varphi = \Sigma_n \Phi$ for every $n$.
Since the theorem is assumed to hold for $f^l$, and $\Phi$ is a continuous function,
\begin{equation}\label{eq.LD}
\mu\left(\left\{x\in M: \left|\frac 1m \Sigma_m\Phi(x)\right| >\alpha\right\}\right))
\le C_\alpha e^{-c_\alpha m} \text{ for every } m\geq 1.
\end{equation}
Given $\epsilon>0$, let $x\in M$ and $n \ge 1$ be such that
\begin{equation}\label{eq.S_ineq}
\left|\frac 1n S_n \varphi(x) \right|>\epsilon.
\end{equation}
Writing $n=ml+r$ with $0 \le r < l$, we get
$$
S_n\varphi(x)
= S_{ml}\varphi(x) + S_r \varphi(f^{ml}(x))
= \Sigma_m\Phi(x) + S_r \varphi(f^{ml}(x))
$$
Then 
$$
\begin{aligned}
\left|\frac 1m \Sigma_m \Phi(x) \right|
& = \frac nm \left|\frac 1n S_n\varphi(x) - \frac 1n S_r \varphi(f^{ml}(x)) \right| \\
& \geq \frac nm\left\{\left|\frac 1n S_n\varphi(x) \right| - \frac {l}{n} \|\varphi\|_{C^0}\right\}
\geq \frac nm\left\{\epsilon -\frac {l}n\|\varphi\|_{C^0}\right\}.
\end{aligned}
$$
For $n\ge 2l\|\varphi_0\|/\epsilon$ this implies that 
$$
\left|\frac 1m \Sigma_m \Phi(x) \right| \ge \frac{l\epsilon}{2}.
$$
Then \eqref{eq.LD} gives that the measure of the set of points $x$ as in \eqref{eq.S_ineq} is bounded by
$$
C_{l\epsilon/2}e^{-c_{l\epsilon/2}m}
$$
which is bounded above by $\tilde{C}_\epsilon e^{-\tilde{c}_\epsilon n}$ for suitable
choices of $\tilde{C}_\epsilon$ and $\tilde{c}_\epsilon$.
The cases $n < 2l\|\varphi_0\|/\epsilon$ are handled by increasing $\tilde{C}_\epsilon$ if necessary.
\end{proof}

\begin{remark}\label{r.iterate_necessary}
Once Theorem~\ref{t.large_deviations} is proved, it will follow that it holds also for every iterate $f^l$,
$l\ge 1$ of a map $f$ as in the statement. That is because the assumptions hold for $f^l$ if
they hold for $f$, as observed at the beginning of the proof of the previous lemma.
\end{remark}

Up to replacing $f$ with $f^{n_1}$, which is allowed by Lemma~\ref{l.iterate_sufficient},
it is no restriction to suppose that the integer $n_1$ in Corollary~\ref{c.discrete} is
equal to $1$. We do so in what follows.

\begin{corollary}\label{c.notlog}
For any $\alpha>0$ and any continuous function $\varphi:M \to\RR$ with $\int \varphi \, d\mu \leq -\alpha$,
there are $s_1>0$ and $\theta_1 \in (0,1)$ such that
$$
\sum_j c_j \exp\left(s \max_{f^{-1}(\xi^u_{i_j}(x_j))} S_n\varphi \right)
\leq \theta_1^{s} \text{ for every } s \in [0,s_1].
$$
\end{corollary}

\begin{proof}
Consider the function
$$
g: s \mapsto \log \sum_j c_j \exp\left(s \max_{f^{-1}(\xi^u_{i_j}(x_j))} S_n\varphi \right)
$$
and observe that $g(0)=0$,
$$
g'(0) = \sum_j c_j \max_{f^{-1}(\xi^u_{i_j}(x_j))} S_n\varphi \leq - \alpha_1
$$
and the second derivative $g''$ is bounded on $[0,1]$, uniformly in $i$ and $x$
(because the $c_j$s take only finitely many values, and the number of terms in
the sum is also uniformly bounded).
\end{proof}

\begin{corollary}\label{c.largedeviation}
For any $\alpha>0$ and any continuous function $\varphi:M \to\RR$ with $\int \varphi \, d\mu \leq -\alpha$, we have
\begin{equation}\label{eq.largedeviation}
\sum_j c_j\exp\left(s_1\max_{f^{-n}(\xi^u_{i_j}(x_j))} S_n\varphi \right)
\leq \theta_1^{n s_1}
\end{equation}
for any $n \ge 1$, $i=1, \dots, k$, and $x \in \Lambda \cap \cM_i$.
\end{corollary}

\begin{proof}
The argument is by induction on $n$. The first step was done in Corollary~\ref{c.notlog}.
Assume that \eqref{eq.largedeviation} hold for $n$.
Write each $f(\xi^u_{i_j}(x_j))$ as a union of plaques $\xi^u_{i_j,m}(x_{j,m})$,
and denote
$$
b_{i_j,m} = \nu^u_{i_j,x_j}\left(f^{-1}(\xi^u_{i_j,m}(x_{j,m}))\right).
$$
Here $m$ varies on some finite set which depends on $i_j$ and $x_j$, but whose cardinal
is uniformly bounded. Keep in mind that $\sum_m b_{i_j,m}=1$ for every $i_j$ and $x_j$.
Then
$$
f^{n+1}(\xi^u_i(x)) = \bigcup_{j} \bigcup_m \xi^u_{i_j,m}(x_{j,m}).
$$
Moreover,
$$
\max_{f^{-(n+1)}(\xi^u_{i_j,m}(x_{j,m}))} S_{n+1}\varphi
\leq
\max_{f^{-n}(\xi^u_{i_j}(x_{j}))} S_{n}\varphi
+
\max_{f^{-1}(\xi^u_{i_j,m}(x_{j,m}))}\varphi
$$
and so
$$
\begin{aligned}
& \sum_j \sum_m c_j b_{i_j,m} \exp\left( s_1 \max_{f^{-(n+1)}(\xi^u_{i_j,m}(x_{j,m}))} S_{n+1}\varphi\right)\\
& \leq
\sum_j c_j \exp\left(s_1 \max_{f^{-n}(\xi^u_{i_j}(x_{j}))} S_{n}\varphi\right)
\sum_m b_{i_j,m} \exp\left(s_1 \max_{f^{-1}(\xi^u_{i_j,m}(x_{j,m}))}\varphi\right).
\end{aligned}
$$
By Corollary~\ref{c.notlog}, the last factor is bounded by $\theta_1^{s_1}$.
Thus, using the induction hypothesis,
$$
\sum_j \sum_m c_j b_{i_j,m} \exp\left( s_1 \max_{f^{-(n+1)}(\xi^u_{i_j,m}(x_{j,m}))} S_{n+1}\varphi\right)\\
\leq \theta_1^{(n+1)s_1}
$$
as we wanted to prove.
\end{proof}

We are ready to prove the large deviations property for $(f,\mu)$.
In fact we prove a slightly more general estimate \eqref{eq.more_general},
valid for any probability measure $\zeta \in E(0)$.

Consider any continuous function $\varphi:M\to\RR$ with $\int_M \varphi \, d\mu=0$.
For any $\alpha>0$, define $\varphi_\alpha=\varphi-\alpha$.
By Lemma~\ref{l.largesteplog} and Corollary~\ref{c.largedeviation} there is
$\theta_\alpha \in (0,1)$ such that
$$
\sum_j c_j \exp\Big(s_1 \max_{f^{-n}(\xi^u_{i_j}(x_j))} S_n\varphi_\alpha\Big)
\leq \theta_\alpha^n,
$$
for every $x \in \cM_i \cap \Lambda$ and $i=1, \dots, k$.
Clearly, $S_n \varphi_\alpha = S_n \varphi - n\alpha$.
Thus, the previous inequality implies that
$$
\int_{\Lambda} \exp\Big(s_1 (S_n\varphi - n\alpha)\Big) \, d\nu^u_{i,x}
\le \theta_\alpha^n
$$
for every $x \in \cM_i \cap \Lambda$ and $i=1, \dots, k$.
In view of the definition of $E(0)$, the inequality extends to every $\zeta\in E(0)$:
$$
\int_{\Lambda} \exp\left(s_1 (S_n\varphi - n\alpha)\right) \, d\zeta
\le \theta_\alpha^n.
$$
Then, by the Chebyshev inequality,
$$
\zeta\left(\left\{x: S_n\varphi \ge n\alpha\right\}\right)
\le  \theta_\alpha^n.
$$
Applying the same argument to $-\varphi$,
we also get that $\zeta\left(\left\{x: S_n\varphi \le - n\alpha\right\}\right)
\le  \theta_\alpha^n$. Thus,
\begin{equation}\label{eq.more_general}
\zeta\left(\left\{x: \left|S_n\varphi\right| \ge n\alpha\right\}\right)
\le 2\theta_\alpha^n
\end{equation}
for any $\zeta\in E(0)$. In particular, this holds for $\zeta=\mu$,
which proves Theorem~\ref{t.large_deviations}.

\section{Proof of the coupling lemma}\label{s.coupling2}

Here we prove Lemma~\ref{l.coupling}. Throughout, we keep the assumptions of Theorem~\ref{main.B}.
In particular, $\Lambda=\supp\mu$ is taken to be connected and, consequently (by \cite[Theorem~A]{UVYY1})), $u$-minimal.

\subsection{Preparing the coupling argument}

We start with the following fact:

\begin{proposition}\label{p.mc}
There are $n_0 \ge 1$ and $\lambda_0<0$ such that
\begin{equation}\label{eq.mc}
\int_{\xi_i^u(x)} \frac{1}{n_0}\log\|Df^{n_0}\mid_{E^{cs}}\| \, d\nu^u_{i,x}
\le \lambda_0
\end{equation}
for every $x\in \cM_i\cap \Lambda$ and every $i=1, \dots, k$.
\end{proposition}

\begin{proof}
Since $f$ is assumed to have $c$-mostly contracting center,
Proposition~\ref{p.contracting} ensures that the center-stable Lyapunov exponents of $\mu$ are all negative.
Let $\lcs<0$ denote the largest of these exponents. Then
\begin{equation}\label{eq.less_lambda}
\int_{\Lambda} \frac{1}{n}\log\|Df^{n}\mid_{E^{cs}}\|\,d\mu < \frac{\lcs}{2}.
\end{equation}
for every large $n$. Take $\lambda_0=\lcs/2$ and let $n$ be fixed such that \eqref{eq.less_lambda} holds.
Consider any sequences $i_j \in \{1, \dots, k\}$, $x_j\in \cM_{i_j} \cap \Lambda$ and $m_j \to \infty$.
Clearly,
\begin{equation}\label{eq.limsupmeasure}
\begin{aligned}
\int_{\xi^u_{i_j}(x_j)} \frac{1}{m_j n} & \log\|Df^{m_j n}\mid_{E^{cs}}\| \, d\nu^u_{i_j,x_j} \\
& \le \frac{1}{m_j} \int_{\xi^u_{i_j}(x_j)}
\frac{1}{n} \sum_{i=0}^{m_j-1} \log\|Df^n\mid_{E^{cs}} \circ f^{in}\| \, d\nu^u_{i_j,x_j}\\
& = \int_{\xi^u_{i_n}(x_n)}\frac{1}{n}\log \|Df^n\mid_{E^{cs}}\|\,d\left(\frac{1}{m_j}
\sum_{i=0}^{m_j-1} (f^{in})_* \nu^u_{i_j,x_j}\right)
\end{aligned}
\end{equation}
Observe that
$$
\frac{1}{m_j}\sum_{i=0}^{m_j-1} (f^{i n})_*\nu^u_{i_j,x_j} \to \mu
$$
as $j$ goes to infinity.
Indeed, \cite[Proposition~4.1]{UVYY1} gives that every accumulation point of this sequence belongs to $E(0)$
and is an invariant probability measure, and so Lemma~\ref{l.Eunique} implies that every such accumulation
point coincides with $\mu$. In view of \eqref{eq.limsupmeasure}, this implies that every accumulation point of
$$
\int_{\xi^u_{i_j}(x_j)} \frac{1}{m_j n} \log\|Df^{m_j n}\mid_{E^{cs}}\| \, d\nu^u_{i_j,x_j}
$$ 
when $j$ goes to infinity is bounded above by
$$
\int_{\Lambda}  \frac{1}{n}\log \|Df^n \mid_{E^{cs}}\| \, d\mu < \lambda_0.
$$
This proves that there exists $m \ge 1$ such that
$$
\int_{\xi^u_{i}(x)}\frac{1}{m n} \log\|Df^{m n}\mid_{E^{cs}}\| \, d\nu^u_{i,x} < \lambda_0
$$
for every $x\in\cM_i\cap\Lambda$ and every $i=1, \dots, k$. Take $n_0= m n$.
\end{proof}

It is not difficult to check that if the conclusion of Lemma~\ref{l.coupling} holds for the iterate $f^{n_0}$,
with maps $\tau_0:Y_1 \to Y_2$ and $R_0:Y_1\to\NN$, then the corresponding statement holds for the
original map $f$ as well, with functions $\tau=\tau_0$ and $R=n_0R_0$, up to suitable changes of the
constants $C_1, C_2$ and $\rho_1, \rho_2$. Thus, it is no restriction to assume that $n_0=1$, and so
\begin{equation}\label{eq.Birkhopf}
\int_{\Lambda} \log \|Df\mid_{E^{cs}}\| \, d\nu^u_{i,x} < \lambda_0 < 0
\end{equation}
for all $x\in \cM_i\cap \Lambda$ and $i=1, \dots, k$. We do that in the following.

Write each $f^n(\xi_i^u(x))$, $ n\ge 1$ as a finite union of strong-unstable plaques $\xi^u_{i_j}(x_j)$,
and then denote
$$
c_j=c_j(i,x,n) = \nu^u_{i,x}(f^{-n}(\xi^u_{i_j}(x_j))).
$$
Applying Corollary~\ref{c.largedeviation} to the function $\Phi =\log \|Df \mid_{E^{cs}}\|$,
we find constants $s_1>0$ and $\theta_1\in(0,1)$ such that
\begin{equation}\label{eq.exponentialtail}
\sum_j c_j \prod_{t=0}^{n-1} \max_{{f^{t-n}(\xi^u_i(x_j))}} \|Df\mid_{E^{cs}}\|^{s_1}\leq \theta_1^{s_1 n}.
\end{equation}
for any $n \ge 1$.

Fix $\lambda<0$ such that
\begin{equation}\label{eq.lambda_def}
\lambda \geq \max\{\lambda_0/2, \log\theta_1/2\}.
\end{equation}
Let $K>1$ and denote by $U_i^n(x) \subset \xi_i^u(x)$ the union of of all the pre-images $f^{-n}(\xi_{i_j}^u(x_j))$
for which
\begin{equation}\label{eq.USet}
\prod_{t=0}^{n-1} \max_{{f^{t-n}(\xi^u_i(x_j))}} \|Df\mid_{E^{cs}}\| > K e^{\lambda n}.
\end{equation}
From \eqref{eq.exponentialtail} and \eqref{eq.USet} we get the Chebyshev-type inequality
$$
\nu^u_{i,x}\left(U_i^n(x)\right) K^{s_1} e^{\lambda s_1 n } \le \theta_1^{s_1 n}.
$$
In view of the choice of $\lambda$, this shows that 
\begin{equation}\label{eq.KU1}
\nu^u_{i,x}\left(U_i^n(x)\right) \le K^{-s_1} e^{\lambda s_1 n}
\end{equation}
for every $n\ge 1$, $x\in\cM_i\cap\Lambda$ and $i=1, \dots, k$.
Up to fixing $K>1$ sufficiently large, the latter implies that there exists $q_1<1$ such that
\begin{equation}\label{eq.KU2}
\nu^u_{i,x}\left(U_i(x)\right) \le q_1
\end{equation}
for every $x\in\cM_i\cap\Lambda$ and $i=1, \dots, k$, where $U_i(x)$ denotes the union of
$U_i^n(x)$ over all $n\ge 1$.

We close this section with the following useful fact, which we quote from
Alves, Bonatti, Viana~\cite[Lemma~2.7]{ABV00}, see also Dolgopyat~\cite[Lemma 8.1]{Dol00}.
Let $\epsilon>0$ be fixed such that, just by continuity,
$$
d(y,z)\leq\epsilon \quad\Rightarrow\quad
\|(Df \mid_{E^{cs}})(y)\|\leq e^{-\lambda/2}\|(Df\mid_{E^{cs}})(z)\|.
$$

\begin{lemma}\label{l.stablemanifold}
If $x\in M$ and $n \geq 1$ are such that
$\|(Df^j \mid_{E^{cs}})(x)\| \leq K e^{\lambda j}$ for $j=1, \dots, n$, then
$$
f^j(\cF^{cs}_{\epsilon}(x)) \subset \cF^{cs}_{r_j}(f^j(x))
$$
for every $0 \leq j \leq n$, where $r_j = K \epsilon e^{\lambda j/2}$.
\end{lemma}

After these preparations, we move to prove Lemma~\ref{l.coupling}.
Let $\cY$ be the set of rectangles $Y=\xi^u_i(x)\times J$ with $x\in \cM_i\cap \Lambda$,
$i=1,\cdots, k$ and $J \subset [0,1]$, endowed with the measures $m_i=\nu^u_{i,x}\times dt$.
Write $f(x,t)=(f(x),t)$.
Recall that in the statement of Lemma~\ref{l.coupling} the sets $Y_1, Y_2 \in\cY$ and the measures
$m_1, m_2$ were taken to satisfy $m_1(Y_1)=m_2(Y_2)$.
We are going to describe an algorithmic construction of maps $\tau$ and $R$ as in the
statement of the lemma.

This algorithm will be presented in a recursive form.
In the first run (to be detailed in Section~\ref{ss.algorithm1}), we will define a stopping time $s(y)$
for the points $y\in Y_1$ where the coupling map $\tau$ has not yet been defined, in such a way that the sets
\begin{equation}\label{eq.Pjn}
P_j^n=\{y\in Y_j: s(y)=n\}, \quad j = 1, 2
\end{equation}
are of the form  $f^{-n}(\bigcup_m Y_{j,n,m})$, where $Y_{j,n,m}= \xi^u_{i_{n,m}}(x_{j,n,m})\times I_{j,n,m}$
are elements of $\cY$ satisfying $m_1(Y_{1,n,m})=m_2(Y_{2,n,m})$.
Finally, we will set $P^\infty_j=Y_j\setminus \cup_n P^n_j$ for $j=1, 2$,
and we will define the function $R$ on the set $P_1^\infty$, and the map $\tau$ from $P^\infty_1$ and $P^\infty_2$.

The purpose of the inductive runs of the algorithm is to extend the domains of $R$ and $\tau$ successively 
to include almost every point of $Y$. This is actually similar to the first run, and will be detailed
in Section~\ref{ss.algorithm2}.

\subsection{First run of the algorithm}\label{ss.algorithm1}

Let us now detail the first run of the algorithm. 
Recall that $\cH^{cs}:\xi^u_i(y_1) \to \xi^u_i(y_2)$
denotes the center-stable holonomy between the strong-unstable plaques of points $y_1$ and $y_2$ in the
same Markov set $\cM_i$. Denote by $d_{cs}$ the distance along center-stable leaves.

\begin{lemma}\label{l.y1y2}
Given $\epsilon>0$ as in Lemma~\ref{l.stablemanifold}, there is $n_0 \ge 1$ such that for any two points
$x_1\in \Lambda \cap \cM_{i_1}$ and $x_2 \in \Lambda \cap \cM_{i_2}$ the iterates $f^{n_0}(\xi^u_{i_1}(x_1))$
and $f^{n_0}(\xi^u_{i_2}(x_2))$ contain strong-unstable plaques $\xi^u_1(y_1)$ and $\xi^u_1(y_2)$ inside the
Markov domain $\cM_1$ and
$$
d_{cs}\left(w,\cH^{cs}_{y_1,y_2}(w)\right)\leq \epsilon
\text{ for any $w \in \xi^u_1(y_1)$.}
$$
\end{lemma}

\begin{proof}
Since $\mu$ is a $c$-Gibbs $u$-state, its push-forward under the map $\pi:M\to\TT^d$ is the Lebesgue
measure on the torus (see \cite[Corollary~3.5]{UVYY1}). In particular, the image of $\Lambda$
under $\pi$ is the whole $\TT^d$.
Fix any point $q$ in the interior of the Markov set $\cR_1 \subset \TT^d$, and let $z\in\pi^{-1}(q)\cap \Lambda$.
Note that $z$ is in the interior of $\cM_1$.
Fix a neighborhood $V$ of $z$ contained in $\cM_1$ small enough that the distance along center-stable leaves
$$
d_{cs}\left(w,\cH^{cs}_{w_1,w_2}(w)\right) \leq \epsilon
\text{ for any $w_1, w_2 \in V$ and any $w \in \xi^u_1(w_1)$.}
$$
Since $\Lambda$ is $u$-minimal, there is $n_0\ge 1$ such that the iterates $f^{n_0}(\xi^u_{i_1}(x_1))$ and $f^{n_0}(\xi^u_{i_2}(x_2))$
of any $x_1\in \Lambda \cap \cM_{i_1}$ and $x_2 \in \Lambda \cap \cM_{i_2}$ 
intersect $V$. Just take $y_1 \in f^{n_0}(\xi^u_{i_1}(x_1)) \cap V$ and $y_2 \in f^{n_0}(\xi^u_{i_2}(x_2)) \cap V$.
\end{proof}

Consider $Y_1 = \xi^u_{i_1}(x_1) \times [0,a]$ and $Y_2 = \xi^u_{i_2}(x_2) \times [0,a]$ for any
$x_1\in \Lambda \cap \cM_{i_1}$, $x_2 \in \Lambda \cap \cM_{i_2}$, and $a \in [0,1]$.
Take $y_1, y_2$ as in Lemma~\ref{l.y1y2}. For $j=1, 2$, define
\begin{equation}\label{eq.overlinec}
\overline{c}_j=\nu^u_{i_j,x_j}(f^{-n_0}(\xi^u_1(y_{j})))
\end{equation}
for $j=1, 2$. It follows from Remark~\ref{r.lowboundary} that these $\overline{c}_j$ take only finitely
many values. Define $\overline{Y}_j=\xi^u_{1}(y_j)\times [0,\overline{t}_j]$, where
\begin{equation}\label{eq.tjbar}
(\overline{t}_1,\overline{t}_2)
= \left\{\begin{array}{ll} (a{\overline{c}_2}/{\overline{c}_1},a) & \text{ if } \overline{c}_2\leq \overline{c}_1 \\
(a,a{\overline{c}_1}/{\overline{c}_2}) & \text{ if }\overline{c}_1\leq \overline{c}_2
\end{array}\right.
\end{equation}
This choice ensures that $f^{-n_0}(\overline{Y}_j)\subset Y_j$ for $j=1, 2$, and
\begin{equation}\label{eq.same_measure}
m_1\left(f^{-n_0}(\overline{Y}_1)\right) = m_2\left(f^{-n_0}(\overline{Y}_2)\right).
\end{equation}
We denote this value as $b$.
On the complements $P_j^{n_0} = Y_j\setminus f^{-n_0}(\overline{Y}_j)$, we define the stopping time $s(y,t)=n_0$.


Let us check that the $P_j^{n_0}$ constructed in this way are indeed of the form described in \eqref{eq.Pjn}.
Due to the Markov property, each $Y_j$ is a union of finitely many sets of the form
$$
f^{-n_0}(Z) = \eta \times [0,a] \text{ with } Z \in \cY.
$$
The total $m_j$-measure is equal to $a$, of course.
By construction, $f^{-n_0}(\overline{Y}_j)$ is a set of the form $\eta \times [0,\overline{t}_j]$
such that the two $m_j$-measures are the same.
Thus, $P_j^{n_0} = Y_j \setminus f^{-n_0}(\overline{Y}_j)$ is a finite union of sets
of the form $\eta \times J$, where either $J = [0,a]$ or $J=(\overline{t}_j,a]$
(note that one of the $\overline{t}_j$ is equal to $a$, and so in the latter the corresponding $J$ is empty).
It is clear that $m_1(P_1^{n_0}) = m_2(P_2^{n_0})$.
Thus, up to cutting (in the vertical direction only) each $\eta \times J$ into finitely many pieces,
in a suitable way, we may write
\begin{equation}\label{eq.Pjnzero}
P_j^{n_0} = \bigcup_m f^{-n_0}(Y_{j,n_0,m})
\end{equation}
with $Y_{j,n_0,m}\in\cY$ satisfying $m_1(Y_{1,n_0,m})=m_2(Y_{2,n_0,m})$ for every $m$.

Next, for $n>n_0$ we define
\begin{equation}\label{eq.P1n}
P^n_1 = f^{-n_0}\left(V^n(y_1) \times [0,\overline{t}_1]\right)
\text{ and }
P^n_2 = f^{-n_0}\left(H^{cs}_{y_1,y_2}\left(V^n(y_1)\right)\times [0,\overline{t}_2]\right),
\end{equation}
where
\begin{equation}\label{eq.Vny1}
V^n(y_1) = U^{n-n_0}_1(y_1)\setminus \bigcup_{m=0}^{n-n_0-1}U^{m}_1(y_1)
\end{equation}
and $U_1^m(y_j)$ is as defined in \eqref{eq.USet}.
It is clear that the $V^n(y_1)$, $n > n_0$ are pairwise disjoint and their union
coincides with
$$
U_1(y_1) = \bigcup_{n\ge 1} U_1^n(y_1).
$$
Thus the $P_j^n$, $n > n_0$ are pairwise disjoint subsets of $Y_j$,
all with the same height, and
$$
\bigcup_{n>n_0} P_1^n
= f^{-n_0}\left(\bigcup_{n > n_0} V^n(y_1) \times [0,\overline{t}_1]\right)
= f^{-n_0}\left(U_1(y_1) \times [0,\overline{t}_1]\right).
$$
We also define
$$
P_j^\infty = Y_j \setminus \bigcup_{n \ge n_0} P^n_j
$$
for $j=1, 2$. Notice that
\begin{equation}\label{eq.P1infty}
\begin{aligned}
P_1^\infty
= Y_1 \setminus \bigcup_{n \ge n_0} P^n_1
& = \left(Y_1 \setminus P_{n_0}\right) \setminus \bigcup_{n > n_0} P^n_1\\
& = f^{-n_0}(\overline{Y}_1) \setminus f^{-n_0}\left(U_1(y_1) \times [0,\overline{t}_1]\right) \\
& = f^{-n_0}\left((\xi^u_1(y_1)\setminus U_1(y_1))\times [0,\overline{t}_1]\right),
\end{aligned}
\end{equation}
and, similarly,
\begin{equation}\label{eq.P2infty}
P_2^\infty = f^{-n_0}\left(H^{cs}_{y_1,y_2}\big((\xi^u_1(y_1)\setminus U_1(y_1))\big)\times [0,\overline{t}_2]\right).
\end{equation}

A few other simple facts about the sequences $P_j^n$ are collected in the next lemma.
Let $K$, $\lambda$, $s_1$, $q_1$, and $b$ be fixed as in \eqref{eq.USet}, \eqref{eq.lambda_def},
\eqref{eq.exponentialtail}, \eqref{eq.KU2}, and \eqref{eq.same_measure}, respectively.

\begin{lemma}\label{l.same_measure}
For every $n > n_0$ we have:
\begin{itemize}
\item[(1)] $m_1(P^n_1)=m_2(P^n_2)\leq b e^{\lambda s_1 (n-n_0)}$;

\item[(2)] $m_1(P^\infty_1)=m_2(P^\infty_2) \geq b (1-q_1)$;

\item[(3)] $f^n(P^n_j)$ may be written as a union $\bigcup_k Y_{j,n,k}$ of elements $Y_{j,n,k}$ of $\cY$
with
\begin{equation}\label{eq.matching}
m_1(f^{-n}(Y_{1,n,k})) = m_2(f^{-n}(Y_{2,n,k}));
\end{equation}
\end{itemize}
\end{lemma}

\begin{proof}
By the definitions of $m_1=\nu^u_{1,x_1} \times dt$ and $P_1^n$ (see \eqref{eq.P1n}),
$$
m_1(P_1^n) 
= m_1(f^{-n_0}(V^{n}(y_1) \times [0,\overline{t}_1]))
= (\nu^u_{i_1,x_1} \times dt)(f^{-n_0}(V^{n}(y_1) \times [0,\overline{t}_1])).
$$
Recall that $\overline{Y}_1=\xi^u_1(y_1)\times [0,\overline{t}_1]$ and 
$m_1(f^{-n_0}(\overline{Y}_1))=b$, by \eqref{eq.same_measure}. 
Thus
$$
\begin{aligned}
\frac{m_1(P_1^n)}{b}
& = \frac{m_1(P_1^n)}{m_1(f^{-n_0}(\overline{Y}_1))} 
= \frac{(\nu^u_{i_1,x_1} \times dt)(f^{-n_0}(V^{n}(y_1) \times [0,\overline{t}_1]))}
{(\nu^u_{i_1,x_1} \times dt)(f^{-n_0}(\overline{Y}_1))}\\
& = \frac{\nu^u_{i_1,x_1} (f^{-n_0}(V^{n}(y_1)))}{\nu^u_{i_1,x_1}(f^{-n_0}(\xi^u_1(y_1)))}.
\end{aligned}
$$
By the definition \eqref{eq.Vny1}, $V^{n}(y_1))$ is a subset of $\xi^u_1(y_1)$.
Thus, as the reference measures have constant Jacobian for $f$,
\begin{equation}\label{eq.simila}
\frac{m_1(P_1^n)}{b} = \frac{\nu^u_{1,y_1} (V^{n}(y_1))}{\nu^u_{1,y_1}(\xi^u_1(y_1))}
= \nu^u_{1,y_1}(V^{n}(y_1)).
\end{equation}
Similarly,
$$
\frac{m_2(P_2^n)}{b} = \nu^u_{1,y_2}(H^{cs}_{y_1,y_2}(V^n(y_1)))
$$
As the reference measures are invariant under center-stable holonomies, by \eqref{eq_cs-invariant},
this implies that $m_2(P_2^n) = m_1(P_1^n)$.
It is also clear from \eqref{eq.Vny1} that $V^{n}(y_1))$ is contained in $U_1^{n-n_0}$.
Thus, \eqref{eq.simila} together with \eqref{eq.KU1} yield
$$
m_1(P_1^n) \le b \nu^u_{1,y_1}(U_1^{n-n_0}(y_1)) 
\leq b K^{-s_1}e^{\lambda s_1 (n-n_0)}
\leq b e^{\lambda s_1 (n-n_0)}.
$$
This completes the proof of claim (1).

Next we prove (2). Since
$$
m_j(P_j^\infty) = m_j(Y_j) - \sum_{j=n_0}^\infty m_j(P_j^n)
= b - \sum_{j=n_0}^\infty m_j(P_j^n),
$$
it follows from the previous remarks that $m_1(P_1^\infty)=m_2(P_2^\infty)$.
By definition, $P_1^\infty = f^{-n_0}(\xi^u_1(y_1) \setminus U_1(y_1))\times [0,\overline{t}_1]$.
Thus, similarly to \eqref{eq.simila},
$$
\frac{m_1(P_1^\infty)}{b} 
= \nu^u_{1, y_1} \left(\xi^u_1(y_1)\setminus U_1(y_1)\right).
$$
By \eqref{eq.KU2}, this yields
$$
m_1(P_1^\infty) \geq b (1-q_1).
$$
This proves claim (2).

By the definition \eqref{eq.USet}, $U_1^{n-n_0}(y_1)$ consists of domains that are mapped by $f^{n-n_0}$
to entire strong-unstable plaques. By the Markov property, it follows that the image of $U_1^m(y_1)$
under $f^{n-n_0}$ consists of entire strong-unstable plaques for every $1 \le m \le n-n_0$.
Therefore, the set $V^n(y_1)$ defined in \eqref{eq.Vny1} is also a union of domains whose images under
$f^{n-n_0}$ are entire plaques.
Using the Markov property once more, we see that the same is true for the image $H^{cs}_{y_1,y_2}(V^n(y_1))$
under the center-stable holonomy. Thus, both $P_j^n$, $j=1, 2$ may be written as unions of sets of the form
$f^{-n}(Y_{j,n,m})$, where $Y_{j,n,m}$ is an element of $\cY$ with height $\overline{t}_j$.
Moreover, the images $f^{r-n}(Y_{1,n,m})$ and $f^{r-n}(Y_{1,n,m})$ are in the same Markov domain for each
$r\in\{n_0, \dots, n\}$, and the center-stable holonomy induces a bijection between them.

We claim that $m_1(f^{-n}(Y_1,n,m)) = m_2(f^{-n}(Y_2,n,m))$ for every $m$. To see this, write
$$
f^{-n}(Y_{j,n,m})=Z_{j,n,m} \times [0,\overline{t}_j]
\text{ with } Z_{j,n,m}\subset f^{-n_0}(\xi^u_1(y_j)).
$$
Then the claim may be rephrased as
\begin{equation}\label{eq.tttt}
\nu^u_{i_1,x_1}(Z_{1,n,m}) \, \overline{t}_1
=
\nu^u_{i_2,x_2}(Z_{2,n,m}) \, \overline{t}_2,
\end{equation}
Using the definition \eqref{eq.overlinec}, together with the fact that the Jacobians of
the reference measures are locally constant, we find that
$$
\frac{\nu^u_{i_j,x_j}(Z_{j,n,m})}{\overline{c}_j}
= \frac{\nu^u_{i_j,x_j}(Z_{j,n,m})}{\nu^u_{i_j,x_j}(f^{-n_0}(\xi^u_1(y_j)))}
= \frac{\nu^u_{1,y_j}(f^{n_0}(Z_{j,n,m}))}{\nu^u_{1,y_j}(\xi^u_1(y_j))}
= \nu^u_{1,y_j}(f^{n_0}(Z_{j,n,m})).
$$
Since the reference measures $\nu^u_{1,y_1}$ and $\nu^u_{1,y_2}$ are mapped to one
another by the center-stable holonomy $H^{cs}_{y_1,y_2}$, we also have that
$$
\nu^u_{1,y_1}(f^{n_0}(Z_{1,n,m})) = \nu^u_{1,y_2}(f^{n_0}(Z_{2,n,m})).
$$
It follows that
$$
\nu^u_{i_1,x_1}(Z_{1,n,m}) \, \overline{c}_2
= \nu^u_{i_2,x_2}(Z_{2,n,m}) \, \overline{c}_1.
$$
This gives \eqref{eq.tttt}, because the definition \eqref{eq.tjbar} is such that $\overline{c}_1\overline{t}_2 = \overline{c}_2\overline{t}_1$. This finishes the proof of claim (3).
\end{proof}

At this point, we define $\tau:P^\infty_1\to P^\infty_2$ and $R:P^\infty_1\to\NN$ in the following way
(keep the expressions \eqref{eq.P1infty} and \eqref{eq.P2infty} in mind).
For any $(x,t) \in P^\infty_1$, let
\begin{equation}\label{eq.definitioncoupling}
\tau(x,t)=\left(y,\frac{\overline{c}_1}{\overline{c}_2} t\right) \text{  and   } R(x,t)=n_0,
\end{equation}
where $y\in\xi^u_{i_2}(x_2)$ is defined by
\begin{equation}\label{eq.xtoy}
y = f^{-n_0} \circ H^{cs}_{y_1,y_2} \circ f^{n_0} (x).
\end{equation}
Let us check the properties in Lemma~\ref{l.coupling} are indeed satisfied at this stage:

\begin{lemma}\label{l.tauvolumpreserving}
Let $(y,s)=\tau(x,t)$ be as in \eqref{eq.definitioncoupling} and \eqref{eq.xtoy}, and $r_n$ be as defined in Lemma~\ref{l.stablemanifold}.
Then
\begin{enumerate}
\item $d(f^n(x),f^n(y))\leq r_{n-n_0}$ for any $n\geq n_0$;
\item $\tau$ maps $m_1$ restricted to $P_1^\infty$ to $m_2$ restricted to $P_2^\infty$.
\end{enumerate}
\end{lemma}

\begin{proof}
By construction, $f^{n_0}(x) \in \xi^u_1(y_1) \setminus U_1(y_1)$.
By the definition \eqref{eq.USet}, it follows $f^{n_0}(x)$ satisfies the assumption of Lemma~\ref{l.stablemanifold}
for every positive iterate.
By Lemma~\ref{l.y1y2}, our choice of $y_1$ and $y_2$ ensures that $f^{n_0}(y) = H^{cs}_{y_1,y_2}(f^{n_0}(x))$
belongs to $\cF^{cs}_\epsilon(f^{n_0}(x))$. Now claim (1) of the present lemma is contained in the conclusion
of Lemma~\ref{l.stablemanifold}.

Next, we claim that the Jacobian of $\tau$ with respect to the measures $m_1$ and $m_2$ is constant.
Since $m_1(P_1^\infty) = m_2(P_2^\infty)$, it follows that the Jacobian is actually equal to $1$, 
which is precisely the content of (2).
To prove the claim, observe that the $m_j=\nu^u_{i_j,x_j}\times dt$, $j=1, 2$ are product measures,
and $\tau$ is a product map.
The Jacobian of the first-variable map $x \mapsto y$ with respect to the reference measures is constant,
since the maps $f^{n_0}$ and $f^{-n_0}$ have locally constant Jacobians, by \eqref{eq_constant_Jacobian},
and the Jacobian of the holonomy map $H^{cs}_{y_1,y_2}$ is constant equal to $1$, by \eqref{eq_cs-invariant}.
The Jacobian of the second variable map $t \mapsto s$, with respect to the Lebesgue measure $dt$, is clearly also constant.
Thus, the overall Jacobian of $\tau$ is constant, as claimed.
\end{proof}

This finishes the first stage of the coupling algorithm.
At this stage, the coupling map $\tau$ is defined between the $P_1^\infty$ and $P_2^\infty$,
and the function $R$ is defined on $P_1^\infty$.

\subsection{Inductive set of the algorithm}\label{ss.algorithm2}

Next, we want to extend the definitions of $\tau$ and $R$ to (full measure subsets of)
the complements $Y_j\setminus P^\infty_j$. This will be done recursively, in the following way.

For each $h \ge 1$, we denote by $T^h$ the subset of $Y_1$ where $\tau$ and $R$ are still
undefined at the end of stage $h$. Thus $T^1 = Y_1 \setminus P_1^\infty = \cup_n P_1^n$.
By induction, we may assume that there are sets 
$$
T^h = \bigcup P_1^{N_h}, \qquad N_h=(n_1, \dots, n_h)
$$
where each $P_j^{N_h}$, $j=1, 2$ is itself a union of sets of the form
$$
f^{-|N_h|}(Y_{j,N_h,m}), \qquad |N_h| = n_1 + \cdots + n_h
$$
with $Y_{j,N_h,m}\in\cY$ for $n_1, \dots, n_h \ge n_0$, and
$$
\begin{aligned}
& m_{1}\left(f^{-|N_h|}(Y_{1,N_h,m})\right)
=
m_{2}\left(f^{-|N_h|}(Y_{2,N_h,m})\right).
\end{aligned}
$$

Applying the first run (Lemma~\ref{l.same_measure}) of the algorithm to each $Y_{j,N_h,m}$,
we find subsets $P^\infty_{j,N_h,m}$ of $Y_j$ and measure-preserving maps
$$
\tau_{N_h,m}: P^\infty_{1,N_h,m} \to P^\infty_{2,N_h,m}
$$
as in the previous section. Then we extend $\tau$ and $R$ to each
$$
f^{-|N_h|}(P^\infty_{1,N_h,m})
\subset f^{-|N_h|}(Y_{j,N_h,m})
\subset T^h
$$
through
$$
\tau = f^{-|N_h|} \circ \tau_{N_h,m} \circ f^{N_h}
\text{ and }
R = |N_h| + n_0.
$$
A key point is that, according to part (2) of Lemma~\ref{l.same_measure}, each
$P^\infty_{1,N_h,m}$ contains a fraction $ \geq 1-q_1$ of the
measure of $Y_{j,N_h,m}$.
Moreover, the proportion is preserved under the backward image,
because the map $f^{n_1+\cdots+n_h}$ has constant Jacobian.
Thus, the measure of the set
$$
T^{h+1}=T^h \setminus \bigcup_{N_h, m} f^{-|N_h|}(P^\infty_{1,N_h,m})
$$
satisfies
\begin{equation}\label{eq.Thplusone}
m_1(T^{h+1})\le q_1 m_1(T^h).
\end{equation}

As a direct application of Lemma~\ref{l.tauvolumpreserving}, we get:

\begin{corollary}~\label{c.tauvolumpreserving}
Let $(y,s)=\tau(x,t)$ for $(x,t)\in P^\infty_{1,N_h, m}$,
and let $r_n$ be as in Lemma~\ref{l.stablemanifold}. Then
\begin{enumerate}
\item $d(f^n(x),f^n(y))\leq r_{n-R((x,t))}$ for any $n\geq R(x,t)$;
\item $\tau$ maps the measure $m_{1}$ restricted to $P^\infty_{1, N_h, m}$ to the measure $m_{2}$ restricted to $P^\infty_{2, N_h, m}$.
\end{enumerate}
\end{corollary}

The construction in the previous section also gives that
$$
\bigcup_m Y_{j,N_h, m} \setminus P^\infty_{j,N_h,m}
= \bigcup_{n'} f^{|N_h|}(P_{j}^{N_h, n'})
$$
such that each $P_j^{N_h,n'}$, $j=1, 2$ is a union of sets of the form
$$
f^{-|N_h| - n'}(Y_{j,N_h,n', m'})
$$
with $Y_{j,N_h, n', m'}\in\cY$ for $n' \ge n_0$, and
$$
m_1\left(f^{-|N_h| - n'}(Y_{1,N_h, n', m'})\right)
= m_2\left(f^{- |N_h| - n'}(Y_{2,N_h, n', m'})\right).
$$
Thus we recover the recursive assumptions for the set $T_{h+1}$.

From \eqref{eq.Thplusone} we get that $\tau$ and $R$ are eventually defined at
$m_1$-almost every point of $Y_1$. Part (1) of Lemma~\ref{l.coupling} is given by
Lemma~\ref{l.tauvolumpreserving} and Corollary~\ref{c.tauvolumpreserving}.
We are left to checking part (2) of the lemma.

Recall that, at each stage $h$ the function $R$ is defined by $R = |N_{h-1}| + n_0$.
Fix some small $\delta>0$. For each $n$, write the set $\{R=n\}$ as the disjoint union of
two subsets, depending on whether $h\le \delta n$  or $h > \delta n$.
It is clear that the latter subset (corresponding to $h > \delta n$), is contained in
$T_{\lfloor \delta n \rfloor}$. Hence, by \eqref{eq.Thplusone}, its $m_1$-measure is bounded by
\begin{equation}\label{eq.expdecay1}
m_1(T_{\lfloor \delta n \rfloor}) \le q_1^{\lfloor \delta n \rfloor}.
\end{equation}
On the other hand, the $m_1$-measure of the former subset (corresponding to $h \le \delta n$)
is given by
\begin{equation}\label{eq.pre-Stirling}
\sum_{\substack{ h, (k_1,\cdots, k_{h-1}) \\ k_1 + \cdots + k_{h-1} + n_0 =n \\ k_1, \dots, k_{h-1} \geq n_0}} m_1(\{n_i=k_i\})
\leq \sum_{\substack{ h, (k_1,\cdots, k_{h-1}) \\ k_1 + \cdots + k_{h-1} + n_0 =n \\ k_1, \dots, k_{h-1} \ge n_0}} \prod_{i=1}^{h-1} b e^{\lambda s_1 (k_i-n_0)}
\end{equation}
This inequality follows inductively from Lemma~\ref{l.same_measure} applied to each run $i=1, \dots, h$,
together with the observation that the pull-back map $f^{-(k_1 + \cdots + k_{i-1})}$ has constant Jacobian.
The factor $b \in (0,1)$, which was defined in \eqref{eq.same_measure}, arises in each of the runs, 
according to claim (1) in Lemma~\ref{l.same_measure}.

We bound the right hand side of \eqref{eq.pre-Stirling} as follows. To begin with, there are $\delta n$ values of $h$.
Moreover, for each fixed value of $h$, the number of terms in the sum is bounded above by
\begin{equation}\label{eq.Stirling}
\begin{aligned}
\frac{(n-(h+1)n_0+h-1)!}{(n-(h+1)n_0)!(h-1)!}
& \le \frac{(n+h-1)!}{n!(h-1)!}\\
& \approx \left(1+\frac{h}{n}\right)^n \left(1+ \frac{n}{h}\right)^h
= \left(\left(1+\frac{h}{n}\right) \left(1+ \frac{n}{h}\right)^{\frac{h}{n}}\right)^n
\end{aligned}
\end{equation}
(check \cite[Corollary 6.7]{BoV00} for a similar estimate using Stirling's formula).
Recall that we are considering $h \le \delta n$, and observe that $(1+1/x)^x \to 1$ when $x\to 0$.
Thus we see that, given any $\epsilon>0$, the right hand side of \eqref{eq.Stirling} is bounded
by $C e^{\epsilon n}$ if $\delta$ is chosen small enough, where $C$ is an absolute constant.
Thus, the total number of terms on the right hand side of \eqref{eq.pre-Stirling} is bounded by
$C e^{\epsilon n} (\delta n)$.
From these remarks, we get that the right hand side of \eqref{eq.pre-Stirling} is bounded above by
\begin{equation}\label{eq.expdecay2}
Ce^{2\epsilon n} (\delta n) e^{\lambda s_1 n}.
\end{equation}

Combining \eqref{eq.expdecay1} and \eqref{eq.expdecay2}, and keeping in mind that $q_1<1$ and $\lambda<0$,
we get that
$$
m_1(\{R=n\}) \le q_1^{\lfloor \delta n\rfloor} + C \delta n e^{(\lambda s_1+2\epsilon)n}
$$
decays exponentially fast with $n$, as long as we choose $\epsilon$ small enough.
Then, clearly, $m_1(\{R > n\})$ also decays exponentially fast with $n$,
as claimed in part (2) of Lemma~\ref{l.coupling}.

This completes the proof of Lemma~\ref{l.coupling}.

\section{Proof of Theorem~\ref{main.C}}\label{s.proofC}

Let $\gamma$ be a positive number. We denote by $C^\gamma(M)$ the Banach space of
$\gamma$-H\"older functions $\varphi:M\to\RR$ with the norm
$$
\|\varphi\|_\gamma = \sup_{x\in M}|\varphi(x)|
+ \sup_{x_1 \neq x_2} \frac{|\varphi(x_1) - \varphi(x_2)|}{d(x_1,x_2)^\gamma}.
$$
We use a similar notation $\|\zeta\|_\gamma$ to denote the operator norm of an element
of the dual space $(C^\gamma(M))^*$, that is, a linear functional $\zeta:C^\gamma(M) \to \RR$.
Every probability measure on $M$ may be viewed as an element of this dual space,
and we will often do that in what follows.

The push-forward operator $f_*$ extends to a linear operator on the whole space $(C^\gamma(M))^*$,
which we still denote as $f_*$, defined by
$$
f_*\zeta:C^\gamma(M) \to \RR, \quad f_*\zeta(\varphi) = \zeta(\varphi \circ f).
$$
This extension $f_*:(C^\gamma(M))^* \to (C^\gamma(M))^*$ is a bounded linear operator:
having fixed any Lipschitz constant $L>1$ for $f$, we have that
\begin{equation}\label{eq.pushforward}
\|f_*\zeta\|_\gamma
= \sup_{\|\varphi\|_\gamma=1} |\zeta(\varphi \circ f)|
\le \|\zeta\|_\gamma \sup_{\|\varphi\|_\gamma=1}  \|\varphi \circ f\|_\gamma
\le \|\zeta\|_\gamma L^\gamma
\end{equation}
for any $\zeta \in (C^\gamma(M))^*$.

In what follows we use the sets $C(R)$ and $E(R)$, $R \ge 0$ introduced in Section~\ref{ss.Holder_densities}:
in a few words, $C(R)$ is the set of probability measures on individual strong-unstable plaques $\xi_i^u(x)$,
$x \in\cM_i \cap \Lambda$ obtained by multiplying the corresponding reference measure $\nu^u_{i,x}$
by some density $e^\rho$ where $\rho$ is $(R,\gamma)$-H\"older;
and $E(R)$ is the space of probability measures on $\Lambda$, not necessarily $f$-invariant, which
are convex combinations, not necessarily finite, of elements of $C(R)$.
In particular, $C(R)$ is a subset of $E(R)$. Their restrictions to each Markov element $\cM_i$ are
denoted $C_i(R)$ and $E_i(R)$, respectively.

\begin{lemma}\label{c.converging}
There exist $C_3>0$ and $\rho_3<1$ such that
$$
\|f_*^n(\zeta_1-\zeta_2)\|_\gamma\leq C_3\rho_3^n
$$
for any $n \ge 1$ and any $\zeta_1, \zeta_2 \in E(0)$.
\end{lemma}

\begin{proof}
Assume first that $\zeta_1, \zeta_2 \in C(0)$, that is, they are of the form $\zeta_j = \nu^u_{i_j,x_j}$
with $x_j \in \cM_{i_j}$. Denote $m_j = \zeta_j \times dt$.
Then, given any $\varphi\in C^\gamma(M)$ and any $n \ge 1$,
$$
\begin{aligned}
(f_*^n \zeta_j)(\varphi)
=\int_{\xi^u_{i_j,x_j}} \varphi(f^n(y)) \, d\nu^u_{i_j,x_j}(y)
= \int_{Y_j} \varphi(f^n(y)) \, dm_j(y,t).
\end{aligned}
$$
for $j=1, 2$, and so,
$$
\begin{aligned}
f_*^n(\zeta_1-\zeta_2)(\varphi)
& = \int_{Y_1} \varphi \circ f^n \, dm_1 -
\int_{Y_2} \varphi \circ f^n \, dm_2\\
& = \int_{Y_1} \left[\varphi \circ f^n - \varphi \circ f^n \circ \tau\right] \, dm_1,
\end{aligned}
$$
where $\tau:Y_1 \to Y_2$ is as in Lemma~\ref{l.coupling}.
Define $Z(n)=\{(y,t) \in Y_1: R(y,t) \leq {n}/{2}\}$.
Then, using both parts of Lemma~\ref{l.coupling} for $n/2$,
$$
\begin{aligned}
\left|f_*^n(\zeta_1-\zeta_2)(\varphi)\right|
& \leq \int_{Z(n)} \left|\varphi \circ f^n-\varphi \circ f^n \circ \tau\right|  \, dm_1
+ 2 \|\varphi\|_0 \, m_1(Y_1 \setminus Z(n)) \\
& \leq \|\varphi\|_\gamma (C_1\rho_1^{{n}/{2}})^\gamma
+ 2 \|\varphi\|_0 C_2 \rho_2^{{n}/{2}}
\le (C_3/2) \rho_3^n \|\varphi\|_\gamma
\end{aligned}
$$
for suitable choices of $C_3$ and $\rho_3$, depending only
on $C_1$, $C_2$, $\rho_1$, $\rho_2$, and $\gamma$.

Now consider the case where $\zeta_1\in C(0)$ and $\zeta_2\in E(0)$.
By definition $\zeta_2$ is a convex combination of measures in $E_i(0)$,
$i=1,\cdots, k$, and so it is no restriction to suppose that $\zeta_2\in E_i(0)$ for some $i$.
By Lemma~\ref{l.E_disintegration}, the disintegration
$$
\zeta_2 = \int_{\xi^u_i} \zeta_P \, d\tilde{\zeta}_2(P)
$$
of $\zeta_2$ with respect to the partition $\xi_i^u$ is such that $\zeta_P \in C_i(0)$
for every $P\in\xi_i^u$. Then,
$$
(f_*^n \zeta_2)(\varphi) =\int_{\xi^u_i} (f_*^n \zeta_P)(\varphi) \, d\tilde{\zeta}_2(P)
$$
for any $\varphi\in C^\gamma(M)$ and $n \ge 1$. So,
$$
\begin{aligned}
\left|f_*^n(\zeta_1-\zeta_2)(\varphi)\right|
& = \left|\int_{\xi^u_i} \big[(f_*^n \zeta_1)(\varphi)-(f_*^n \zeta_P)(\varphi)\big] \, d\tilde{\zeta}_2(P)\right|\\
& \le \int_{\xi^u_i} \big|f_*^n (\zeta_1-\zeta_P)(\varphi)\big| \, d\tilde{\zeta}_2(P)
\le (C_3/2) \rho_3^n \|\varphi\|_\gamma.\\
\end{aligned}
$$

Finally, for any $\zeta_1$ and $\zeta_2$ in $E(0)$, we may pick any $\zeta_3\in C(0)$ and use the triangle inequality
together with the previous paragraph to conclude that
$$
|f_*^n(\zeta_1-\zeta_2)(\varphi)| \le C_3 \rho_3^n\|\varphi\|_\gamma
$$
for any $\varphi\in C^\gamma(M)$ and $n \ge 1$.
\end{proof}

This enables us to prove that the push-forwards of any measure $l \in E(0)$ under the map $f$
converge exponentially fast to $\mu$ relative to the norm $\|\cdot \|_\gamma$:

\begin{corollary}
For any $\zeta\in E(0)$ and $n \ge 1$,
$$
\|f_*^n \zeta - \mu \|_\gamma \leq C_3\rho_3^n.
$$
\end{corollary}

\begin{proof}
By Lemma~\ref{l.Eunique}, the invariant measure $\mu$ belongs to $E(0)$.
Thus this is a special case of the previous lemma.
\end{proof}

Proceeding with the proof of Theorem~\ref{main.C}, we now extend this analysis to measures in $E(R)$ for any $R>0$:

\begin{lemma}\label{l.aproximate}
For any $R>0$ and any $\zeta\in E(R)$ there exists $\zeta_0 \in E(0)$
such that $\|f^n_*\zeta - f^n_*\zeta_0\|_\gamma \leq R e^R$ for any $n\geq 1$.
\end{lemma}

\begin{proof}
By definition, every $\zeta \in E(R)$ is a convex combination of elements of $E_i(R)$, $i=1, \dots, k$.
So, it is no restriction to assume that $ \zeta \in E_i(R)$ for some $i$.
By Lemma~\ref{l.E_disintegration}, we may write
$$
\zeta = \int_{\cM_i} e^{\rho_x} \nu^u_{i,x} \, d\hat\zeta(x)
$$
where $\hat\zeta$ is a probability measure on $C_i(R)$, and each function $\rho_x$ satisfies
\begin{equation}\label{eq.rho_normalized}
\int_{\xi^u_i(x)} e^{\rho_x} \, d\nu^u_{i,x} =1
\end{equation}
together with the H\"older condition \eqref{eq.rho_Holder}.
Let us check that
$$
\zeta_0 = \int_{\cM_i} \nu^u_{i,x} \, d\hat{\zeta}(x)
$$
satisfies the claim.
It is no restriction to assume that the diameters of all $\xi_i^u(x)$ are bounded by $1$,
and then \eqref{eq.rho_Holder} implies that
\begin{equation*}
e^{-R} \leq e^{\rho_x(y)-\rho_x(z)} \leq e^R
\text{ for every $y, z\in \xi_i^u(x)$ and $x\in\cM_i$.}
\end{equation*}
Property \eqref{eq.rho_normalized} implies that the minimum (respectively, maximum) of $e^{\rho_x}$
on $\xi^u_i(x)$ is less (respectively, greater) than or equal to $1$.
So, the previous inequality also yields that
$$
e^{-R} \leq e^{\rho_x(y)} \leq e^R
\text{ for every $y\in\xi_i^u(x)$ and $x\in\cM_i$.}
$$
In particular, $|e^{\rho_x}-1| \le R e^R$ for every $x$. Then, for any $\varphi\in C^\gamma(M)$,
\begin{equation*}
\begin{aligned}
\left|\int\varphi \, df^n_*\zeta - \int \varphi \, df^n_*\zeta_0\right|
& = \left|\int \left[ \int \varphi\circ f^n \left(e^{\rho_x} -1\right) \, d\nu_{i,x}^u \right] d\hat\zeta(x)\right|\\
& \leq \|\varphi\circ f^n\|_0 R e^R
\leq \|\varphi\|_\gamma R e^R.
\end{aligned}
\end{equation*}
This gives the claim.
\end{proof}

Going back to the proof of Theorem~\ref{main.C}, consider any $\zeta \in E(R)$. Let $\omega<1$ be as in \eqref{eq.omega}.
By Proposition~\ref{p.Econtracting}
$$
f_*^m \zeta \in E\left(R e^{l \gamma m\log \omega}\right)
\text{ for any $m\ge 1$.}
$$
Replacing $\zeta$ and $R$ with $f^m_*\zeta$ and $R e^{l \gamma m\log \omega}$ in Lemma~\ref{l.aproximate},
we find that for each $m\ge 1$ there exists $\zeta_m \in E(0)$ such that
$$
\|f_*^k(f_*^m \zeta - \zeta_m) \|_\gamma
\leq R e^{l \gamma m\log \omega} \exp\left(R e^{l \gamma m\log \omega}\right)
\text{ for any $k\ge 1$.}
$$
Given any $m\ge 1$, let $k, m \approx n/2$ such that $k+m=n$. Then
$$
\begin{aligned}
\|f_*^n\zeta - \mu\|_\gamma 
& \le \|f_*^k(f_*^m \zeta - \zeta_m)\|_\gamma + \|f_*^k\zeta_m-\mu\|_\gamma \\
& \le R \omega^{l \gamma m} \exp\left(R \omega^{l \gamma m}\right) + C_2 \rho_3^k 
\end{aligned}
$$
Take $\tau=\max\{\omega^{l\gamma/2},\rho_3^{1/2}\}$, and note that it is in $(0,1)$,
since $\omega$ and $\rho_3$ are. Moreover, the previous inequality ensures that
\begin{equation}\label{eq_almost_there}
\|f_*^n\zeta - \mu\|_\gamma \le L \tau^n \text{ for every $n\ge 1$}
\end{equation}
if $L >0$ is chosen suitably.

Finally, $\psi$ be any $\gamma$-H\"older function not identically zero.
Then $\Psi=\psi+2\|\psi\|_0$ is a strictly positive function and it is still $\gamma$-H\"older.
More to the point, $\log\Psi$ is also $\gamma$-H\"older.
Let $R$ be the multiplicative H\"older constant. Then the probability measure
$$
\zeta=\frac{\Psi \mu}{\int_M \Psi\,d\mu}
= e^{\log\Psi -\log \int_M \Psi\,d\mu} \mu
$$
is in $E(R)$, since $\mu$ is in $E(0)$. Since the difference $\psi-\Psi$ is constant, the correlation
$$
\left|\int_M (\varphi\circ f^{n})\psi \, d\mu - \int_M \varphi \, d\mu \int_M \psi \, d\mu \right|
$$
is not affected if we replace $\psi$ with $\Psi$. So \eqref{eq_almost_there} gives that
$$
\begin{aligned}
\left|\int_M (\varphi\circ f^{n})\psi \, d\mu - \int_M \varphi \, d\mu \int_M \psi \, d\mu \right|
& = \left|\int_M (\varphi\circ f^{n}) \Psi \, d\mu - \int_M \varphi \, d\mu \int_M \Psi \, d\mu \right| \\
& 
= \int_M \Psi \, d\mu \left|\int_M (\varphi\circ f^{n}) \, d\zeta - \int_M \varphi \, d\mu \right| \\
&  \le \|f_*^n\zeta - \mu\|_\gamma \|\varphi\|_\gamma \|\Psi\|_0 \\
& \le L \tau ^n \|\varphi\|_\gamma \|\psi\|_0.
\end{aligned}
$$
Just take $K(\varphi,\psi) = L \|\varphi_\gamma \|\psi\|_0$. The proof of Theorem~\ref{main.C} is complete.

\bibliographystyle{alpha}
\bibliography{bib}

\end{document}